\documentclass[10pt,reqno]{article}

\usepackage{amsthm}
\usepackage{amsmath,amstext,amssymb,amsfonts}
\usepackage{mathrsfs}
\usepackage{graphicx}
\usepackage{multirow}
\usepackage{diagbox}
\usepackage[left=3cm,right=3cm,top=3.5cm,bottom=4cm,marginparwidth=1.7in]{geometry}

\newtheorem{example}{Example}[section]
\newtheorem{theorem}{Theorem}[section]
\newtheorem{lemma}[theorem]{Lemma}

 \def\ss{\smallskip}
\def\m{\mbox}   
\newcommand{\q}{\quad}   \newcommand{\qq}{\qquad} 
\def\no{\noindent}

\newcommand{\pa}{\partial}

\newcommand{\lla}{\|{\hskip -1pt}|}
\newcommand{\rra}{\|{\hskip -1pt}|}

\newcommand{\al}{\alpha}

\newcommand{\Om}{\Omega}
\newcommand{\de}{\delta}

\newcommand{\lam}{\lambda}

\newcommand{\vep}{\varepsilon}
\newcommand{\lj}{|{\hskip -1pt} \|}
\newcommand{\rj}{|{\hskip -1pt} \|}

\newcommand{\diam}{\mathrm{diam}}

\newcommand{\cM}{{\cal M}}

\newcommand{\E}{{\mathbb{E}}}

\newcommand{\sd}{\mathsf{d}}

\newcommand{\R}{{\mathbf{R}}}

\newcommand{\be}{\begin{eqnarray}}
\newcommand{\ee}{\end{eqnarray}}
\newcommand{\beq}{\begin{equation}}
\newcommand{\eeq}{\end{equation}}
\newcommand{\ben}{\begin{eqnarray*}}
\newcommand{\een}{\end{eqnarray*}}
\newcommand{\nn}{\nonumber}
\newtheorem{assu}{Assumption}
\newtheorem{alg}{Algorithm}
\numberwithin{assu}{section}
\numberwithin{theorem}{section}
\numberwithin{alg}{section}

\begin{document}

\title{{Stochastic convergence of regularized solutions and their finite element approximations 
to \\ inverse source problems}
}
\author{Zhiming Chen\thanks{\footnotesize LSEC, Institute of Computational Mathematics,
Academy of Mathematics and System Sciences and School of Mathematical Science, University of
Chinese Academy of Sciences, Chinese Academy of Sciences,
Beijing 100190, China. 
 (zmchen@lsec.cc.ac.cn). } \and
Wenlong Zhang\thanks{\footnotesize Department of Mathematics, Southern University of Science and Technology
(SUSTech), Shenzhen, Guangdong Province, P.R.China. 
(zhangwl@sustech.edu.cn).}\and
Jun Zou \thanks{Department of Mathematics, 
The Chinese University of Hong Kong, Shatin, N.T., Hong Kong. 
(zou@math.cuhk.edu.hk).}}

\date{}
\maketitle


\begin{abstract}
In this work, we investigate the regularized solutions and their finite element solutions 
to the inverse source problems governed by partial differential equations, and 
establish the stochastic convergence and optimal finite element convergence rates 
of these solutions,  under pointwise measurement data with random noise. 
Unlike most existing regularization theories, the regularization error estimates are derived without any source 
conditions, while the error estimates of finite element solutions 
show their explicit dependence on the noise level, regularization parameter, mesh size, and time step size, 
which can guide practical choices among these key parameters in real applications. 
The error estimates also suggest an iterative algorithm for determining an optimal regularization parameter. 
Numerical experiments are presented to demonstrate the effectiveness of the analytical results. 
\end{abstract}


%

\medskip
\section{Introduction}
This work presents a quantitative understanding of stochastic convergence of the regularized solutions and 
their finite element approximations to the inverse source problems governed by partial differential equations,  
under the measurement data with random noise. The inverse source 
problems may arise from very different applications and modeling, e.g., 
diffusion or groundwater flow processes \cite{new1,AB01, new3, GER83, new2, new6,Isakov2013}, 
heat conduction or convection-diffusion processes \cite{AB05, new4,GER83, liu16, tadi97}, 
or acoustic problems \cite{Badia2011, nelson20}. 
Pollutant source inversion can find many applications, e.g., 
indoor and outdoor air pollution, detecting and monitoring underground water pollution.
Physical, chemical and biological measures have been developed  for the identification of
sources and source strengths \cite{AB01,ZZFZG11,ZCP01}. 
Due to the important applications of ill-posed inverse source problems, stable numerical solutions have been widely 
studied, both deterministically and statistically \cite{LZ07,SK97, SK95}.
A popular approach for inverse source problems is the least-squares optimization 
with appropriate regularizations \cite{AB05, GER83,WY10}, which will be also the formulation we take in this work.

Our first main result is the establishment of the optimal stochastic error estimates of the regularized solutions 
in terms of the noise level, without any source conditions. 
This presents a brand new idea in error estimates of approximate solutions to ill-posed 
inverse problems achieved by regularization, and 
it is very different from the existing regularization theories and their approximation error estimates 
nearly all of which were established under some source conditions.
Regularization and convergence of regularized solutions have been widely studied
under various source conditions. The classical source condition requires 
the existence of a small source function \cite{engl89}. 
One source condition was proposed in \cite{engl2000} for an inverse conductivity problem
to relax the restrictive requirement on the smallness of the source function in the classical convergence theory 
\cite{engl89}. 
%
A variational source condition was proposed in \cite{HKPS07}, and were further extended 
in \cite{BH10, Fle10, Grasm10}. It is still a hot topic how to verify the classical or variational source conditions 
for most inverse problems under reasonable physical assumptions on the forward solutions and identifying 
parameters. It appears that the analytical techniques in all existing verifications of source conditions 
are quite different for each concrete inverse problem 
\cite{ChenJiangZou20, ChenYousept, Hohage15,Hohage1502,KLW20019}.
The current work makes a very promising first attempt to achieve the error estimates of regularized solutions,  
without any source conditions, hence gets rid of the technical difficulties in convergence analysis.


The second main contribution of this work is to derive the stochastic convergence and 
error estimates of finite element approximations 
to the inverse source problems. 
The error estimates of finite element solutions to inverse problems have been known to be quite challenging 
and still open to most practically important inverse problems. 
There have been various efforts on
error estimates of finite element solutions for inverse problems, especially for inverse elliptic and parabolic equations. 
But most existing studies have been carried out only for some not so frequently used mathematical formulations
of inverse problems; see \cite{zou10} for a detailed review and related references therein. We are not aware 
of any error estimates of finite element solutions to the frequently used least-squares formulations with Tikhonov regularizations, especially when the observation data are treated as random variables. 
We had a recent study in \cite{huzou19} for a modified regularization formulation for 
an inverse stationary source problem, where 
error estimates were achieved under some negative norms, which, however, may be rather inconvenient 
to realize in applications. 
One of our main focuses in this work is to make an attempt to fill the gap, to provide 
error estimates of finite element solutions to the least-squares formulations 
with Tikhonov regularizations, and 
more importantly, the observation data will be treated fully as random variables in the entire analysis. 
As we shall demonstrate, the new error estimates are not only optimal but also presents explicit dependence 
on the critical parameters like noise level, regularization parameter, mesh size and time stepsize.
Results of this type are highly desirable in real applications as they can 
provide explicit guidance in choosing these key parameters, and are also the major challenge and difficulty 
in error estimates of finite element solutions to regualarized inverse problems.  

We would like to mention a very important by-product from our convergence analysis, namely, 
it suggests a deterministic iterative algorithm for finding an effective regularization parameter. 
The choice of an effective regularization 
parameter is essential to the success of all output least-squares minimization approaches with Tikhonov regularizations, but it has remained to be a big challenge how to find an effective regularization parameter
for most inverse problems. 

Another feature of this work is that the entire analysis is carried out for a very practical scenario, i.e., 
the scattered data. We shall assume the measurement data is collected pointwise, with noise, 
otherwise no any additional regularity assumption is made. 
This is unlike analyses and results in most existing regularization theories. 

We studied in a recent work \cite{Chen-Zhang} the stochastic convergence of a nonconforming 
finite element method for the thin plate spline smoother for observational data. The spline model for scattered data has attracted considerable attention in the literature. The convergence rate in expectation of the error between the solution of the spline model and the true solution was established in \cite{Utreras}. Under the condition that measurement noise are sub-Gaussian random variables, the stochastic convergence of the empirical error was obtained by the peeling argument in \cite{Geer}
($d=1$) and \cite{Chen-Zhang} ($d=2,3$). 
We shall borrow some analytical tools from \cite{Chen-Zhang, Utreras} to study 
the stochastic convergence in expectation when the measurement noise is random variables 
having bounded variance in subsection \ref{ssec:bv}. The peeling argument is used in subsection \ref{ssec:gaussian} 
to show that the empirical error has an exponential decaying tail when the measurement noise is sub-Gaussian random variables. The discretization and its error estimates 
are considered in section \ref{sec:discrete} both in the expectation and in the Orilicz norm for sub-Gaussian measurement noise. The general results developed in section \ref{sec:stationary} 
are applied to study an inverse nonstationary source problem in section \ref{sec:heat_sourse}. 
And numerical examples are presented in section \ref{sec:numerics} to demonstrate the effectiveness of our 
analytical results. 

\section{Inverse source problem}\label{sec:stationary}

Let $\Omega$ be a bounded domain in $\mathbf{R}^d$ ($d=1,2,3$), and $X$ and $Y$ be two {real Hilbert spaces} 
such that $Y$ is continuously embedded in $C(\bar\Om)$ and compactly embedded in $L^2(\Om)$. 
The inner product and the norm of a Hilbert space $H$ are denoted as $(\cdot,\cdot)_H$ and $\|\cdot\|_H$, 
respectively; but $(\cdot,\cdot)$ is used if $H=L^2(\Om)$.
Throughout the paper, we shall use $C$, with or without subscript, to denote a generic constant independent of the mesh size $h$, the time step size $\tau$, and it may take a different value at each occurrence.

Let $S$ be a {linear bounded operator} from $X$ to $Y$ and $f^*\in X$ be an unknown source. 
We are interested in the inverse source problem of the general form:

\ss
({\bf SIP})  Given the measurement data of $Sf^*$, recover the source  $f^*$. 

\ss

There are many examples of inverse source problems of this type.
%
Our studies will focus on a very important physical scenario, {assuming that the pointwise} measurement data 
is collected on a set of distributed sensors located at $\{x_i\}^n_{i=1}$ {($x_i\ne x_j$ for $i\ne j$)} inside the physical domain 
$\Om$ \cite{AB05, new4, new2, new6, L08,nelson20, NNR98}. 
{We assume that the measurements come} with noise and takes the form 
\begin{equation}\label{eq:data} 
m_i=(Sf^*)(x_i)+e_i, \q i=1, 2, \cdots, n,
\end{equation} 
where $e=(e_1, e_2, \cdots, e_n)^T$ is the data noise vector, 
with $\{e_i\}^n_{i=1}$ being independent and identically distributed random variables on a probability space ($\mathfrak{X},\mathcal{F},\mathbb{P})$. {We shall denote $m=(m_1, m_2, \cdots, m_n)^T$ to be the vector of scattering data.}  
 {Throughout this work, we write $\mathbb{E}[A]$ for the expectation of a random variable $A$.}

We look for an approximate solution $f_n$ of the unknown source function $f^*$ through 
the least-squares regularized minimization:
\beq\label{p1}
\mathop {\rm min}\limits_{f\in X}\frac{1}{n}\sum\limits_{i=1}^{n} {|(Sf)(x_i)-m_i|^2+\lambda_n \|f\|_{X}^2},
\eeq
where $\lambda_n>0$ is called a regularization parameter.

We shall consider that the set of discrete points $\{x_i\}_{i=1}^n$ are scattered but quasi-uniformly distributed
in $\Omega$, i.e., there exists a constant $B>0$ such that ${d_{\max}}/{d_{\min}} \leq B$, where 
${d_{\max}}$ and ${d_{\min}}$ are defined by 
\beq\label{aa}
d_{\max}=\mathop {\rm sup}\limits_{x\in \Omega} \mathop {\rm inf}\limits_{1 \leq i \leq n} |x-x_i| 
~~~\mbox{and} ~~ ~
d_{\min}=\mathop {\rm inf}\limits_{1 \leq i \neq j \leq n} |x_i-x_j|.
\eeq
For any $u,v\in C(\bar\Omega)$ and $y\in \mathbb R^n$, we define 
$$
(y,v)_n=\frac{1}{n}\sum^n_{i=1}y_iv(x_i), \q 
(u,v)_n=\frac{1}{n}\sum^n_{i=1}u(x_i)v(x_i),
$$ 
and the semi-norm $\|u\|_n=(\sum_{i=1}^{n} u^2(x_i)/n)^{1/2}$ for any $u\in C(\bar\Omega)$.

Throughout the work, we consider two kinds of random noises $\{e_i\}^n_{i=1}$: 

\ss
\no ({\bf R1})  $\{e_i\}^n_{i=1}$ are independent random variables satisfying
$\mathbb{E}[e_i]=0$ and $\mathbb{E}[e^2_i]\leq \sigma^2$;

\no ({\bf R2})  $\{e_i\}^n_{i=1}$ are independent sub-Gaussian random variables with parameter $\sigma$,

\ss
\no and provide two different techniques to analyse the stochastic convergence 
and a practical approach to choose the parameter $\lambda_n$ in each case. 
We study the convergence under the expectation $\mathbb{E}$ in the case {\bf (R1)}, 
and establish a stronger convergence in the case {\bf (R2)}, where 
the errors have exponential decay tails. 

\subsection{Stochastic convergence for noisy data of variables with bounded variance}\label{ssec:bv}
We consider the measurement data of type {\bf (R1)} in this section, 
and study the stochastic convergence of the error under the expectation $\mathbb{E}$. 

\begin{assu}\label{Assumption1} 
We assume that 

(1) There exists a constant $\beta>1$ such that for all $u\in Y$, 
 \beq\label{f2}
\|u\|^2_{L^2(\Omega)}\le C(\|u\|^2_n+ n^{-\beta}\|u\|^2_{Y}),\ \ \|u\|^2_n\le C(\|u\|^2_{L^2(\Omega)}+ n^{-\beta}\|u\|^2_{Y}).
\eeq

(2) The first $n$ eigenvalues, $0<\eta_1\le\eta_2\le\cdots\le\eta_n$, of the eigenvalue problem 
\ben
(\psi ,v)_X=\eta\, (S\psi ,Sv)\ \ \forall v\in X,
\een
satisfy that $\eta_k\ge Ck^{\alpha}$ $(k=1, 2, \cdots, n)$ for some constant $C$ depending only on the operator $S:X\to Y$. The  {constant} $\al$ satisfies $1<\al\le\beta$.
\end{assu} 

The following observation is inspired by \cite{Utreras}, where it was shown 
that the solution of a thin plate spline smoother model is attained in a finite dimensional subset.
\begin{lemma}\label{lem:2.2} For a given $m\in \mathbb R^n$, let $f$ be the solution to the optimization 
problem
\begin{equation}
\min_{f\in X,(Sf)(x_i)=m_i} \|f\|^2_{X},
\end{equation}
then $f\in V_n$, where $V_n$ is an n-dimensional subset of $X$.
\end{lemma} 

\begin{proof} 
Let $V$ be a subset of $X$ such that
\ben
V=\{v\in X: (Sv)(x_i)=0, i=1,2,\cdots, n\}\,.
\een
Define the projection operator $P_V: X \rightarrow V$,
\ben
(P_V [f],v)_X=(f,v)_X\ \  \forall v\in V.
\een
Choose $\phi_i\in X$ such that $(S\phi_i)(x_j)=\delta_{ij}$, where $\de_{ij}$ is the Kronecker delta function.
Let $\psi_i= - P_V[\phi_i] +\phi_i$ and $V_n=\m{span}\{\psi_1, \cdots, \psi_n\}$. {It's easy to check that $(S\psi_i)(x_j)=\delta_{ij}$ also holds.} For any $f\in X$, 
define the interpolation operator $I$:
\ben
If=\sum_{i=1}^n (Sf)(x_i) \psi_i.
\een
We can easily see that $If\in  V_n$ and $f-If\in V$, hence we derive 
\ben
(f-If,If)_X
&=&(f-If,\sum_{i=1}^n (Sf)(x_i)(\phi_i- P_V[\phi_i] ))_X\\
&=&\sum_{i=1}^n (Sf)(x_i)(f-If,\phi_i- P_V[\phi_i])_X=0\,,
\een 
where we have used the fact that $(v,\phi_i- P_V[\phi_i])_X=0$ for all $v\in V$.

We see directly from the above equality that $(If,If)_X\leq (f,f)_X$, hence we have 
\ben
\min_{f\in V_n,(Sf)(x_i)=m_i} \|f\|^2_{X} =\min_{f\in X,Sf(x_i)=m_i} \|f\|^2_{X}.
\een
This completes the proof.  \end{proof}

\begin{lemma}\label{lemma:2.3} {Assume Assumption \ref{Assumption1} is fulfilled.} Let $V_n$ be defined as in Lemma \ref{lem:2.2}, then the eigenvalue problem
\beq\label{eigen-n}
(\psi ,v)_X=\rho \,(S\psi ,Sv)_n ~~\forall \,v\in V_n,
\eeq
has $n$ eigenvalues $\rho_1\le\rho_2\le\cdots\le\rho_n$, and all the eigenfunctions form 
an orthogonal basis of $V_n$ with respect to the norm $\|S\cdot\|_n$. Moreover,
there exists a constant $C>0$ independent of $k$ such that
$\rho_k\ge Ck^{\alpha}$ for $k=1,2,\cdots, n.$
\end{lemma} 
\begin{proof} 
Consider $V_n=\m{span}\{\psi_i\}_{i=1}^n$ as defined in the proof of Lemma \ref{lem:2.2}, 
and $(S\psi_i)(x_j)=\delta_{ij}$. 
We can write $\psi=\sum_{i=1}^n(S\psi)(x_i)\psi_i$ 
for any $\psi\in V_n$. This implies $\|S\cdot\|_n$ is a norm of $V_n$. Therefore, the generalized eigenvalue problem 
\eqref{eigen-n} has $n$ finite eigenvalues $\rho_1\le\rho_2\le\cdots\le\rho_n$ and all eigenfunctions 
form an orthogonal basis of $V_n$ with respect to the norm $\|S\cdot\|_n$. 

We are now ready to give a lower bound of the eigenvalues $\rho_k$. 
Using the min-max principle of the Rayleigh quotient for the eigenvalues and \eqref{f2}, we can derive 
\ben
\rho_k
&=&\min_{dim(X)=k,X\subset V_n}\max_{u\in X} \frac{(u,u)_X}{(Su,Su)_n}\\
&\geq & C \min_{dim(X)=k, X\subset V_n}\max_{u\in X} \frac{(u,u)_X}{(Su,Su) + n^{-\beta}(u,u)_X}\\
&\geq & C \min_{dim(X)=k, X\subset L^2(\Omega)}\max_{u\in X} \frac{(u,u)_X}{(Su,Su) + n^{-\beta}(u,u)_X}\\
&=& C \frac{1}{\eta_k^{-1}+ n^{-\beta}}
\ge C\frac{1}{k^{-\al}+n^{-\beta}},
\een
where we have used the fact that $\eta_k\ge Ck^\al$ by Assumption 1. Now $k^{\al}n^{-\beta}\le n^{\al-\beta}\le 1$ for all $k\leq n$ and $\al\le\beta$. We conclude that $\rho_k\ge Ck^\al$. This completes the proof. 
\end{proof}

\begin{theorem}\label{thm:2.1}
{Assume Assumption \ref{Assumption1} is fulfilled.} Let $f_n\in X$ be the unique solution of (\ref{p1}).
Then there exist constants $\lambda_0 > 0$ and $C>0$ such that for any $\lambda_n \leq \lambda_0$,
\begin{align}
\mathbb{E} \big[\|Sf_n-Sf^*\|^2_n\big] &\leq C \lambda_n \|f^*\|^2_{X} + {C\sigma^2}/({n\lambda^{1/\alpha}_n}),
\label{p5}\\
\mathbb{E} \big[\|f_n-f^*\|^2_{X}\big] &\leq C \|f^*\|^2_{X} + {C\sigma^2}/({n\lambda^{1+1/\alpha}_n}).\label{p6}
\end{align}
{More over if we assume the eigenfunctions $\{\phi_k\}_{k=1}^\infty$ of $S$ form an orthonormal basis of $X$, and define the spase $Z$ as $$Z=\{v\in X:~ v=\sum_{k=1}^\infty v_k\phi_k,~with~ v_k=(v,\phi_k)_{L^2(\Omega)}~ and ~ \sum_{k=1}^\infty \eta_k^{1/2}v_k^2<\infty\}.$$ Then we have the following weaker convergence result for $n^{-\beta}\leq \lambda_n$:
\begin{equation}\label{w-con}
\mathbb{E} \big[\|f_n-f^*\|^2_{Z'}\big] \leq C \lambda^{1/2}_n \|f^*\|^2_{X} + {C\sigma^2}/({n\lambda^{1/2+1/\alpha}_n}),
\end{equation}
where $Z'$ is the dual space of $Z$.}
\end{theorem}

\begin{proof} By deriving the necessary condition of the quadratic minimization (\ref{p1}), we can readily see 
that the unique minimizer $f_n\in X$ satisfies the variational equation
\beq
\label{p7}
\lambda_n (f_n,v)_X+(Sf_n,Sv)_n =(m,Sv)_n\ \ \ \forall v\in X.
\eeq
For any $v\in X$, we introduce the energy norm
$\lla v\rra^2_{\lambda_n}:=\lambda(v,v)_X+\|Sv\|_n^2$.
By taking $v=f_n-f^*$ in \eqref{p7} {along with \eqref{eq:data}}, we obtain 
\beq\label{x4}
\lj f_n-f^*\rj_{\lam_n}\le \lam_n^{1/2} \|f^*\|_{X}+\sup_{v\in L^2(\Om)}\frac{(e,Sv)_n}{\lla v\rra_{\lam_n}}.
\eeq

It remains to estimate the supremum term in \eqref{x4}. Using Lemma \ref{lem:2.2}, 
we can rewrite this supremum term equivalently as 
\ben
\sup_{v\in X}\frac{(e,Sv)^2_n}{\lla v\rra_{\lam_n}^2}
&=&\sup_{v\in X}\frac{(e,Sv)^2_n}{\lambda_n(v,v)_X+\|Sv\|_n^2}\\
&\le & \sup_{v\in X}\frac{(e,Sv)^2_n}{\lambda_n \min_{u\in X, Su(x_i)=Sv(x_i)}(u,u)_X+\|Sv\|_n^2}\\
&= & \sup_{v\in X}\frac{(e,Sv)^2_n}{\lambda_n \min_{u\in V_n, Su(x_i)=Sv(x_i)}(u,u)_X+\|Sv\|_n^2}\\
&= & \sup_{v\in V_n}\frac{(e,Sv)^2_n}{\lambda_n(v,v)_X+\|Sv\|_n^2}\,.\\
\een

Let $\rho_1\le\rho_2\le\cdots\le\rho_n$ be the eigenvalues of the problem
\beq\label{x2}
(\psi,v)_X=\rho(S\psi,Sv)_n\ \ \ \ \forall v\in V_n, 
\eeq
with the corresponding eigenfunctions $\{\psi_k\}^n_{k=1}$, which is an
orthonormal basis of $V_n$ under the inner product $(S\cdot,S\cdot)_n$. Thus 
$(S\psi_k,S\psi_l)_n=\delta_{kl}$ and consequently, $(\psi_k,\psi_l)_X=\rho_k\de_{kl}$, $k,l=1,2,\cdots,n$.
 
Now for any $v\in V_n$, 
we have the expansion $v(x)=\sum^n_{k=1}v_k\psi_k(x)$, where $v_k=(Sv,S\psi_k)_n$ for $k=1,2,\cdots,n$. Thus
$\lla v\rra^2_{\lam_n}=\sum^n_{k=1}(\lam_n\rho_k+1)v_k^2$.
By the Cauchy-Schwarz inequality we can readily get 
\ben
(e,Sv)_n^2 &= & {\frac{1}{n}\sum^n_{i=1}e_i\sum^n_{k=1} v_k \psi_k(x_i) = \frac{1}{n}\sum^n_{k=1}v_k\sum^n_{i=1} e_i \psi_k(x_i) }\\
&\leq &\frac 1{n^2}\sum^n_{k=1}(1+\lam_n\rho_k)v_k^2\cdot\sum^n_{k=1}(1+\lam_n\rho_k)^{-1}\Big(\sum^n_{i=1}e_i (S\psi_k)(x_i)\Big)^2.
\een
This, along with the fact that $\|S\psi_k\|_n=1$, implies
\ben
{\mathbb{E}\Big[\sup_{v\in V_n}\frac{(e,Sv)_n^2}{\lla v\rra^2_{\lam_n}}\Big]}
&\le&\frac 1{n^2}\sum^n_{k=1}(1+\lam_n\rho_k)^{-1}\mathbb{E}\Big(\sum^n_{i=1}e_i (S\psi_k)(x_i)\Big)^2\\
&\leq &\sigma^2n^{-1}\sum^n_{k=1}(1+\lam_n\rho_k)^{-1}.
\een
{In the last inequality, we use the fact that the random variables $\{e_i\}^n_{i=1}$ are independent and identically distributed, i.e. $\mathbb{E}[e_ie_j]=\delta_{ij}$.}
 
Now by Assumption \ref{Assumption1} we readily derive 
\ben
\mathbb{E}\Big[\sup_{v\in X}\frac{(e,Sv)_n^2}{\lla v\rra^2_{\lam_n}}\Big]
&\le&C\sigma^2 n^{-1}\sum_{k=1}^n(1+\lam_nk^{\alpha})^{-1}\le C\frac{\sigma^2}{n\lam_n^{1/\alpha}}\,.
\een
This completes the proof by using \eqref{x4}.

Furthermore, if the eigenfunctions $\{\phi_k\}_{k=1}^\infty$ of $S$ form an orthonormal basis of $X$, i.e. $(\phi_k,\phi_l)=\delta_{kl}$, then $(S\phi_k,S\phi_l)=\eta^{-1}_k\delta_{kl}$. For any $v\in X$, we have the expansion $v=\mathop\sum_{k=1}^\infty v_k\phi_k$ with $v_k=(v,\phi_k)$. Obviously, $\|Sv\|^2_{L^2(\Omega)}=\sum_{k=1}^\infty \eta_k^{-1}v^2_k$ and $\|v\|^2_X=\sum_{k=1}^\infty v^2_k$. By definition of dual space and $\|g\|^2_Z=\sum_{k=1}^\infty \eta_k^{1/2}g_k^2$, 
\ben
\|v\|_{Z'}&=&\sup_{0\ne g\in Z}\frac{|(v,g)_X|}{\|g\|_X}= \sup_{0\ne g\in Z}\frac{|\sum_{k=1}^\infty g_k v_k|}{\|g\|_X}\\
&\leq & \frac{(\sum_{k=1}^\infty \eta_k^{-1/2}g^2_k )^{1/2}(\sum_{k=1}^\infty \eta_k^{1/2}g^2_k )^{1/2}}{\|g\|_X}\\
&=& (\sum_{k=1}^\infty \eta_k^{-1/2}v^2_k )^{1/2}\\
&\leq & (\sum_{k=1}^\infty \eta_k^{-1}v^2_k )^{1/4}(\sum_{k=1}^\infty v^2_k )^{1/4} =\|Sv\|^{1/2}_{L^2(\Omega)}\|v\|^{1/2}_X.
\een
Take $v$ to be $f^*-f_n$ in the above inequality, we derive that,
\ben
\|f^*-f_n\|^2_{Z'}\leq \|Sf^*-Sf_n\|_{L^2(\Omega)}\|f^*-f_n\|_X.
\een
From Assumption \ref{Assumption1} (1), 
$\|Sf^*-Sf_n\|^2_{L^2(\Omega)}\leq \|Sf^*-Sf_n\|^2_n + n^{-\beta}\|f^*-f_n\|^2_X$, along with \eqref{p5} and \eqref{p6}, we finally have,
\begin{equation*}
\mathbb{E} \big[\|f_n-f^*\|^2_{Z'}\big] \leq C(1+\lambda_n^{-1}n^{-\beta})^{1/2} \left(\lambda^{1/2}_n \|f^*\|^2_{X} + {\sigma^2}/({n\lambda^{1/2+1/\alpha}_n})\right).
\end{equation*}
With $n^{-\beta}\leq \lambda_n$, we prove the weaker convergence \eqref{w-con}.
 \end{proof} 

\subsection{Stochastic convergence for noisy data being sub-gaussian random variables}\label{ssec:gaussian}
We consider in this section the case {\bf (R2)} for the data \eqref{eq:data}, 
that is, 
\begin{equation}\label{e1}
\mathbb E\Big[ \m{exp}(\lambda (e_i - \mathbb E[e_i]))\Big] \le \m{exp} \Big(\frac 12 \sigma^2 \lambda^2\Big) 
~~\forall\,\lambda \in \mathbb R\,, 
\end{equation}
and study the stochastic convergence of the error $\|Sf^*-S f_n\|_n$.

We first give a brief introduction of sub-Gaussian random variables and the theory of empirical processes that will be 
used in our subsequent analysis; see \cite{Chen-Zhang,Vaart, Geer} for more details. 
The probability distribution function of a sub-Gaussian random variable $Z$ has an exponentially decaying tail, that is, 
\beq\label{gg1}
\mathbb{P}(|Z-\E [Z]|\ge z)\le 2\,\m{exp}\Big( {-\frac {z^2}{2\sigma^2} \Big)} ~~\forall\, z>0.
\eeq

We shall also use the Orlicz norm. For a monotonically increasing convex function $\psi$ satisfying $\psi(0)=0$, 
the Orilicz norm $\|Z\|_\psi$ of a random variable $Z$ is defined as
\beq\label{e2}
\|Z\|_\psi=\inf\Big\{C>0:\mathbb{E}\Big[\psi\Big(\frac{|X|}C\Big)\Big]\le 1\Big\}.
\eeq
For most of our analyses, we will use the Orlicz norm $\|Z\|_{\psi_2}$,  with
$\psi_2(t)=e^{t^2}-1$ for $t>0$. Through some calculations, we have the estimate (see, e.g., \cite[(4.5)]{Chen-Zhang})
\beq\label{e3}
\mathbb{P}(|Z|\ge z)\le 2\,\m{exp}\Big(-\frac{z^2}{\|Z\|_{\psi_2}^2}\Big)   ~~\forall \,z>0.
\eeq

Consider a semi-metric space $\mathbb{T}$ with a semi-metric $\sd$ and the random process $\{Z_t:t\in \mathbb{T}\}$ 
indexed by $\mathbb{T}$. The random process $\{Z_t:t\in \mathbb{T}\}$ is called sub-Gaussian if
\beq\label{e41}
\mathbb{P}(|Z_s-Z_t|>z)\le 2\,\m{exp}\Big( -\frac{z^2}{2\,\sd(s,t)^2} \Big) ~~\forall \,s,t\in \mathbb{T}, ~~z>0.
\eeq
For a semi-metric space $(\mathbb{T},\sd)$ and $\vep>0$, the covering number $N(\vep,\mathbb{T},\sd)$ is 
the minimum number of $\vep$-balls that cover $\mathbb{T}$; and $\log N(\vep,\mathbb{T},\sd)$ is 
called the covering entropy that is a crucial quantity to characterize the complexity of spce $\mathbb{T}$.  
We assume
\begin{assu}\label{Assumption2} For a unit ball $SY$ in $Y$ and any $\vep>0$, 
there exists a constant $\gamma<2$ such that the covering entropy is controlled by 
\ben
\log N(\vep, SY, \|\cdot\|_{L^{\infty}(\Omega)})\le C\vep^{-\gamma}\,.
\een
%
%
\end{assu} 

Important estimates of the covering entropy for Sobolev spaces can be found in \cite{Birman}.  We shall often need the following maximal inequality \cite[Section 2.2.1]{Vaart}. 
\begin{lemma}\label{lem:4.2}
If $\{Z_t:t\in \mathbb{T}\}$ is a separable sub-Gaussian random process, then it holds for some constant $K>0$ that 
\ben
\|\sup_{s,t\in T}|Z_s-Z_t|\|_{\psi_2}\le K\int^{\diam\, \mathbb{T}}_0\sqrt{\log N\Big(\frac\vep 2,T,\sd\Big)}\ d\vep\,.
\een
\end{lemma}

The useful results in the following two lemmas can be found in \cite{Chen-Zhang}.
\begin{lemma}\label{lem:4.5}
$\{E_n(f):=(e,Sf)_n: f\in X\}$ is a sub-Gaussian random process with respect to the semi-distance $\sd(f,v)=
\sigma n^{-1/2}\|Sf-Sv\|_n$ for any $f, v\in X$. 
\end{lemma}
\begin{lemma}\label{lem:4.6}
Let $C_1>0$ and $K_1>0$ be two constants, and $Z$ be any random variable satisfying 
\ben
\mathbb{P}(|Z|>\alpha (1+z))\le C_1\,{\rm exp} \Big(-\frac{z^2}{K_1^2}\Big) \ \ \forall\,\alpha>0, ~~z\ge 1\,,
\een
then there exists a constant ${C(C_1,K_1)>0}$ depending on $C_1$ and $K_1$ such that 
$$\|Z\|_{\psi_2}\le C(C_1,K_1)\,\alpha\,.$$ 
\end{lemma}

\begin{theorem}\label{thm:4.1}
{Assume Assumption \ref{Assumption2} is fulfilled.} Let $\rho_0=\|f^*\|_{X}+\sigma n^{-1/2}$, and $f_n\in X$ be the solution of the minimization \eqref{p1}. If we take
$\lambda_n^{1/2+\gamma/4}= O(\sigma n^{-1/2}\rho_0^{-1})$,
then there exists a constant $C>0$ such that
\ben
\mathbb{P}(\|Sf_n-Sf^*\|_n\ge \lam_n^{1/2}\rho_0z)\le 2\,e^{-Cz^2} 
\q {\rm and} \q \mathbb{P}(\|f_n\|_{X}\ge \rho_0z)\le 2\,e^{-Cz^2}.
\een
{More over, with the same assumptions and notations in Theorem \ref{thm:2.1}, we have, 
\begin{equation}\label{w-con2}
\mathbb{P}(\|f_n-f^*\|_{Z'}\ge \lambda_n^{1/4}\rho_0z)\le 2\,e^{-Cz^2}.
\end{equation}}
\end{theorem}
\begin{proof} By using the estimate \eqref{e3}, it suffices to prove
\beq\label{xx2}
\|\|Sf_n-Sf^*\|_n\|_{\psi_2}\le C\lam_n^{1/2}\rho_0 \q {\rm and} \q  \|\|f_n\|_X\|_{\psi_2}\le C\rho_0.
\eeq
Because of similarity, we will prove only the first estimate in \eqref{xx2} by the peeling argument. 
It follows from \eqref{p1} that
\beq\label{g1}
\|Sf_n-Sf^*\|^2_n + \lambda_n \|f_n\|_{X}^2 \leq 2(e,Sf_n-Sf^*)_n + \lambda_n \|f^*\|_{X}^2.
\eeq
Let $\delta >0,\ \rho>0$ be two constants to be determined later, and we set for $i,j\ge 1$, 
\beq\label{g2}
A_0=[0,\delta), \q A_i=[2^{i-1}\delta,2 ^i\delta), \q B_0=[0,\rho), \q B_j=[2^{j-1}\rho,2^j\rho)\,.
\eeq
For $i,j\ge 0$, we further define
\ben
F_{ij}= \{v \in X:~  \|Sv\|_n \in A_i ~,~ \|v\|_{X} \in B_j \},
\een
then we can readily see 
\beq\label{g3}
\mathbb{P}(\|Sf_n-Sf^*\|_n>\de)\le\sum_{i=1}^\infty\sum_{j=0}^\infty \mathbb{P}(f_n-f^*\in F_{ij}).
\eeq
Now we estimate $\mathbb{P}(f_n-f^*\in F_{ij})$ for each pair $\{i, j\}$. 
By Lemma \ref{lem:4.5}, we know 
$\{(e,Sv)_n:v\in X\}$ is a sub-Gaussian random process with respect to the semi-distance 
$\sd(f,v)$. With this semi-distance, it is easy to see that 
$
\diam(F_{ij})\le  2\sigma n^{-1/2}\cdot 2^i\de\,, 
$
then we can deduce by using Lemma \ref{lem:4.2} that 
\ben
\|\sup_{f-f^*\in F_{ij}}|(e,Sf-Sf^*)_n|\|_{\psi_2}&\le& K\int^{\sigma n^{-1/2}\cdot 2^{i+1}\de}_0\sqrt{\log N\left(\frac\vep 2,F_{ij}, \sd\right)}\,d\vep\nn\\
&=&K\int^{\sigma n^{-1/2}\cdot 2^{i+1}\de}_0\sqrt{\log N\left(\frac\vep{2\sigma n^{-1/2}},F_{ij}, \|S\cdot\|_n\right)}\,d\vep.
\een
By Assumption\,\ref{Assumption2}, we have the estimate for the covering entropy
\ben
&&\log N\left(\frac\vep{2\sigma n^{-1/2}},F_{ij}, \|S\cdot\|_n\right)\le\log N(\frac\vep{2\sigma n^{-1/2}},F_{ij}, \|S\cdot\|_{L^\infty(\Om)})\\
&=&\log N(\frac\vep{2\sigma n^{-1/2}},S(F_{ij}), \|\cdot\|_{L^\infty(\Om)})
\le C\Big(\frac{2\sigma n^{-1/2}\cdot2^j\rho}{\vep}\Big)^{\gamma},
\een
where we have used the fact that $S(F_{ij})$ is included in the ball in $Y$ of radius $C(2^j\rho)$ since $S:X\to Y$ is a bounded operator.
Using this, we can further derive 
\be
\|\sup_{f-f^*\in F_{ij}}|(e,Sf-Sf^*)_n\|_{\psi_2}&\le&K\int^{\sigma n^{-1/2}\cdot 2^{i+1}\de}_0\Big(\frac{2\sigma n^{-1/2}\cdot2^j\rho}{\vep}\Big)^{\gamma/2}\,d\vep\nn\\
&=&C\sigma n^{-1/2}(2^j\rho)^{\gamma/2}(2^i\de)^{1-\gamma/2}.\label{g5}
\ee
Then by using the estimates \eqref{g1} and \eqref{e3}, we have for $i,j \geq 1$, 
\ben
\mathbb{P}(f_n-f^*\in F_{ij})& \leq & \mathbb{P}\Big(2^{2(i-1)}\delta^2 + \lambda_n 2^{2(j-1)}\rho^2 \leq 2 \mathop {\sup}\limits_{f-f^* \in F_{ij}}|(e,f-f^*)_n|
+ \lambda_n \rho^2_0 \Big) \\
& =& \mathbb{P}\Big(2 \mathop {\sup}\limits_{f-f^* \in F_{ij}}|(e,Sf-Sf^*)_n|\ge 2^{2(i-1)}\delta^2 + \lambda_n 2^{2(j-1)}\rho^2 - \lambda_n \rho^2_0\Big)\\
&\le &2\exp \Big[- \frac{1}{C\sigma^2 n^{-1}} \Big(\frac{2^{2(i-1)}\delta^2 + \lambda_n 2^{2(j-1)}\rho^2-\lambda_n \rho^2_0}{ (2^i\delta)^{1-\gamma/2} (2^j\rho)^{\gamma/2}} \Big)^2\Big].
\een
Now for $z \ge 1$, we take
$ 
\delta^2=\lambda_n\rho_0^2 (1+z)^2$, 
$\rho=\rho_0\,, 
$ 
then with the choice that $\lambda_n^{\frac 12+\frac \gamma 4}=O(\sigma n^{-1/2}\rho_0^{-1})$ 
and direct computing, we readily obtain for $i,j \geq 1$ that 
\beq \label{eq:Pij}
\mathbb{P}(f_n-f^*\in F_{ij}) \le  2\exp \Big[ - C 
\Big(\frac{2^{2(i-1)}z(1+z) + 2^{2(j-1)}}{ (2^i (1+z))^{1-\gamma/2} (2^j)^{\gamma/2}} \Big)^2\Big].
\eeq

To simplify the above estimate, 
we use Young's inequality that $ab\le a^p/p+ b^q/q$ for any $a,b>0$ and $p,q>1$ such that 
$p^{-1}+q^{-1}=1$ to obtain 
$$(2^i (1+z))^{1-\gamma/2} (2^j)^{\gamma/2}\le C( (1+z)2^i+2^j).$$ 
Therefore we get from \eqref{eq:Pij} for $i, j\ge 1$ that 
\ben
\mathbb{P}(f_n-f^*\in F_{ij})\le 2\exp \left[ - C (2^{2i} z^2 + 2^{2j}) \right].
\een
Similarly, one can show for $i\geq 1, j=0$ that 
\ben
\mathbb{P}(f_n-f^*\in F_{i0}) \le 2\exp \left[- C (2^{2i} z^2) \right].
\een
Collecting the above estimates for all $i, j\ge 0$ and using the facts that 
$$
\sum^\infty_{j=1} \m{exp} \big(-C(2^{2j})\big) \le \m{exp} ({-C})< 1\q \m{and} \q  
\sum^\infty_{i=1} \m{exp} \big({-C(2^{2i}z^2)}\big) \le \m{exp} ({-Cz^2}),
$$ 
we come to the conclusion that 
\ben
\sum_{i=1}^\infty\sum_{j=0}^\infty \mathbb{P}(f_n-f^*\in F_{ij})\le2\sum_{i=1}^\infty\sum_{j=1}^\infty 
\m{exp} ({- C (2^{2i} z^2 + 2^{2j})})+2\sum^\infty_{i=1} \m{exp} ({- C (2^{2i} z^2)}). 
\een
The above estimate can be further bounded by $4 \m{exp} ({-Cz^2})$. 
Using this, we get from \eqref{g3} that 
\beq\label{g7}
\mathbb{P}(\|Sf_n-Sf^*\|_n>\lam_n^{1/2}\rho_0(1+z))\le 4 \,\m{exp} ({-Cz^2})\ \ \ \ \forall z\ge 1.
\eeq
This, along with Lemma \ref{lem:4.6},  implies that $\|\|Sf_n-Sf^*\|_n\|_{\psi_2} \leq C \lambda_n^{1/2}\rho_0$, 
which is the first estimate in \eqref{xx2}. The second estimate is similar to the first one by taking $i\geq 0$ and $j\geq 1$ in the summation above \eqref{g7}. Using the very same technique in Theorem \ref{thm:2.1}, one could directly get \eqref{w-con2}.
\end{proof} 

\subsection{Convergence of the discrete solutions}\label{sec:discrete}

In this section we consider the approximation to the optimal control problem \eqref{p1}, i.e., 
\ben
\mathop {\rm min}\limits_{f\in X}         \|Sf-m\|^2_n+\lambda_n \|f\|_{X}^2.
\een
We can directly verify that the solution $f_n\in X$ satisfies the weak formulation
\beq\label{b1}
\lambda_n (f_n,v)_X+(Sf_n,Sv)_n =(m,Sv)_n\ \ \ \ \forall v\in X\,.
\eeq

Let $V_h\subset X$ and $Y_h\subset Y$ be two discrete function spaces (e.g., finite element spaces) 
with dimensions $N_h$ and $M_h$ respectively, and $S_h: X\to Y_h\subset Y$ 
be the discrete approximation of the operator $S:X\to Y$. 
We make the following standard assumptions on the discretization space $V_h$ 
and the approximation operator $S_h$.

\begin{assu}\label{Assumption3} For the discrete operator $S_h: X\to Y_h\subset Y$, 

(1) there exists an error estimate $e(h)$ such that the discrete operator $S_h$ satisfies 
\ben
\|Sf-S_hf\|^2_n\leq Ce(h)\|f\|^2_X \q \forall\,f\in X\,.
\een

(2) For any $f\in X$, there exists $v_h\in V_h$ such that
\ben
\lambda_n \|f-v_h\|^2_{X} + \|S_hf-S_hv_h\|^2_n\leq C(\lambda_n+e(h))\|f\|^2_X.
\een
\end{assu}

We can now look for the discrete solution to the problem \eqref{p1}: 
\ben
\mathop {\rm min}\limits_{f_h\in V_h}         \|S_hf_h-m\|^2_n+\lambda_n \|f_h\|_{X}^2.
\een
Obviously, $f_h$ satisfies the weak formulation:
\beq\label{d1}
\lambda_n (f_h,v_h)_X+(S_hf_h,S_hv_h)_n =(m,S_hv_h)_n\ \ \ \ \forall v_h\in V_h.
\eeq

\subsubsection{Convergence for noisy data from random variables with bounded variance}
We study in this section the expectational convergence of the discrete solution to \eqref{d1} 
in the case {\bf (R1)} for the data \eqref{eq:data}, 
with the main results stated below.
\begin{theorem}\label{thm:3.1}
{Assume Assumption \ref{Assumption1} and \ref{Assumption3} are fulfilled.} Let $f_h\in V_h$ be the solution of \eqref{d1}.
Then there exist constants $\lambda_0 > 0$ and $C>0$ such that for any $\lambda_n \leq \lambda_0$,
\begin{eqnarray}
{ } \qq \mathbb{E}\big[\|Sf^*- S_hf_h\|_n^2\big] &\le& C(\lambda_n+e(h))\|f^*\|^2_{X}+C\Big[1+\frac{e(h)}{\lam_n}+\frac{N_h e(h)}{\lambda^{1-1/\alpha}_n}\Big]\frac{\sigma^2}{n\lambda_n^{1/\alpha}}\,,  \label{d2}\\
{ } \qq  \mathbb{E}\big[\|f^*- f_h\|_{X}^2\big] &\le& C\frac{\lambda_n+e(h)}{\lambda_n}\|f^*\|^2_{X}+C\Big[1+\frac{e(h)}{\lam_n}+\frac{N_h e(h)}{\lambda^{1-1/\alpha}_n}\Big]\frac{\sigma^2}{n\lambda_n^{1+1/\alpha}}\label{d2-2}. 
\end{eqnarray}
%
{More over, with the same assumption and notations in Theorem \ref{thm:2.1}, we have, 
\begin{eqnarray}
\mathbb{E}\big[\|f^*- f_h\|_{Z'}^2\big] &\le & C(\lambda^{1/2}_n+e^{1/2}(h))\frac{\lambda_n+e(h)}{\lambda_n}\|f^*\|^2_{X}\nn\\
&+& C(\lambda^{1/2}_n+e^{1/2}(h))\Big[1+\frac{e(h)}{\lam_n}+\frac{N_h e(h)}{\lambda^{1-1/\alpha}_n}\Big]\frac{\sigma^2}{n\lambda_n^{1/\alpha}}\label{w-con-d1}.  
\end{eqnarray}}
In particular, if $e(h)\le C\lambda_n$ and $N_h e(h)\leq C\lambda_n^{1-1/\alpha}$, we have
\begin{align} 
\mathbb{E}\big[\|Sf^*- S_hf_h\|_n^2\big]&\le C\lambda_n\|f^*\|_{X}^2+{C\sigma^2}/({n\lambda_n^{1/\alpha}}), 
\label{d3}\\
\mathbb{E}\big[\|f^*- f_h\|_{X}^2\big]&\le C\|f^*\|_{X}^2+{C\sigma^2}/({n\lambda_n^{1+1/\alpha}}),\\ 
\mathbb{E}\big[\|f^*- f_h\|_{Z'}^2\big] &\le C\lambda_n^{1/2}\|f^*\|_{X}^2+{C\sigma^2}/({n\lambda_n^{1/2+1/\alpha}}).
\end{align}
%
%
\end{theorem}

\begin{proof}
For any $f,v\in X$, we denote $a_h(f,v)=\lambda_n (f,v)_X+(S_h f,S_h v)_n$ and $\|f\|_{a_h}^2=a_h(f,f)$. For any $w_h\in V_h$, by taking $v=w_h$ in \eqref{b1} and $v_h=w_h$ in \eqref{d1}, we readily obtain
\ben
a_h(f_h-v_h,w_h)&=&a_h(f_n-v_h,w_h)+((S-S_h)f_n,S_hw_h)_n+(Sf^*-Sf_n,(S_h-S)w_h)_n\\
&+&\,(e,(S_h-S)w_h)_n :\equiv a_h(f_n-v_h,w_h) + F(w_h)
\quad \forall v_h,w_h\in V_h.
\een
By the triangle inequality, we can further derive  
\beq
\|f_n-f_h\|_{a_h}\le C\inf_{v_h\in V_h}\|f_n-v_h\|_{a_h}+C\sup_{w_h\in V_h}\frac{|F(w_h)|}{\|w_h\|_{a_h}}\,.\label{pa0}
\eeq
But from Assumption \ref{Assumption3}\,(1), we have
\begin{align}
\sup_{w_h\in V_h} \frac{|((S-S_h) f_n,S_h w_h)_n|}{\|w_h\|_{a_h}} &\leq \|S f_n-S_h f_n\|_n \leq Ce(h)^{1/2}\|f_n\|_{X},
\label{pa1}\\
\sup_{w_h\in V_h} \frac{|(Sf^* -S f_n, (S_h-S)w_h)_n|}{\|w_h\|_{a_h}} &\leq C\|S f^*-S f_n\|_n\frac{e(h)^{1/2}}{\lambda_n^{1/2}}.
\label{pa2}
\end{align}

Now we estimate $\mathbb{E}(\sup_{w_h\in V_h} {|(e,S w_h-S_h w_h)_n|^2}/{\|w_h\|^2_{a_h}})$. Let $\{\psi_k\}_{k=1}^{N_h}$ be the orthogonal basis of $V_h$ (with $N_h=$ dim($V_h$)) such that $(\psi_i,\psi_j)=\delta_{ij}$. 
Then for any $w_h\in V_h$, we have $w_h=\sum_{j=1}^{N_h}(w_h,\psi_j)\psi_j$, and $\|w_h\|^2_{L^2(\Omega)}=\sum_{j=1}^{N_h}(w_h,\psi_j)^2$. Applying the Cauchy-Schwarz inequality,
\ben
(e,(S-S_h)w_h)_n^2 &\leq&\frac{1}{n^2} \sum_{j=1}^{N_h}(w_h,\psi_j)^2 \sum_{j=1}^{N_h} \Big(\sum_{i=1}^{n} e_i (S-S_h)\psi_j(x_i)\Big)^2 \\
&=&\frac{1}{n^2} \|w_h\|^2_{L^2(\Omega)} \sum_{j=1}^{N_h} \Big(\sum_{i=1}^{n} e_i (S-S_h)\psi_j(x_i)\Big)^2, 
\een
we derive 
\be
& &\mathbb{E}\Big(\sup_{w_h\in V_h} \frac{|(e,S w_h-S_h w_h)_n|^2}{\|w_h\|^2_{a_h}}\Big)\leq \frac{1}{\lambda_n n^2}   \sum_{j=1}^{N_h} \mathbb{E}\Big(\sum_{i=1}^{n} e_i (S-S_h)\psi_j(x_i)\Big)^2 \label{yy1}\\
&&\hskip2cm= \frac{1}{\lambda_n n} \sum_{j=1}^{N_h} \sigma^2\|(S-S_h)\psi_j\|^2_n
\leq C\frac{\sigma^2}{\lambda_n n} N_h e(h).\nn
\ee
This completes the desired estimates by substituting \eqref{pa1}, \eqref{pa2}, \eqref{yy1} into \eqref{pa0} and using Assumption \ref{Assumption3}\,(2) and Theorem \ref{thm:2.1}. 

With same notations in Theorem \ref{thm:2.1} and apply the estimate therein, we have that
\ben
\|f^*-f_h\|^2_{Z'} & \leq & \|Sf^*-Sf_h\|_{L^2(\Omega)}\|f^*-f_h\|_X \\
& \leq & C\left(\|Sf^*-Sf_h\|_{n}+n^{-\beta}\|f^*-f_h\|_X\right)\|f^*-f_h\|_X \\
& \leq & C\left(\|Sf^*-S_hf_h\|_{n}+\|S_hf_h-Sf_h\|_{n}+n^{-\beta}\|f^*-f_h\|_X\right)\|f^*-f_h\|_X.
\een
Apply the estimates \eqref{d2}, \eqref{d2-2} and assumption \ref{Assumption3} (1) to the above inequality, we finally have \eqref{w-con-d1}. 
\end{proof}

\subsubsection{Convergence for noisy data being sub-gaussian random variables}
We consider in this subsection the convergence of the discrete solution in the case {\bf (R2)} for the data 
\eqref{eq:data}. 
We start by recalling the following lemma in \cite[Corollary 2.6]{Geer} about the estimation of the covering entropy of finite dimensional subsets.
\begin{lemma}\label{lem:4.7}
Let $G$ be a finite dimensional subspace of $X$ of dimension $N_G>0$ and $G_R=\{f\in G: \|f\|_{X}\le R\}$. Then it holds that 
\ben
N(\vep,G_R,\|\cdot\|_{X})\le (1+4R/\vep)^{N_G}\ \ \forall \vep>0.
\een
\end{lemma}

\begin{lemma}\label{lem:4.8}
{Assume Assumption \ref{Assumption3} is fulfilled.} Let $G_h:=\{w_h\in V_h: \|w_h\|_{a_h}\le 1\}$. Assume that $e(h)\leq C\lam_n$ and $N_h e(h)\leq C\lambda_n^{1-\gamma/2}$. Then it holds that 
\ben
\|\,\sup_{w_h\in G_h}|(e, S w_h-S_h w_h)_n|\,\|_{\psi_2}\le C\sigma n^{-1/2}\lam_n^{-\gamma/4}.
\een
\end{lemma}

\begin{proof} By Lemma \ref{lem:4.5} we know that $\{\hat E_n(v_h):=(e,S w_h-S_h w_h)_n\ \forall w_h\in G_h\}$ 
is a sub-Gaussian random process with respect to the semi-distance $\hat\sd(v_h,w_h)=\sigma n^{-1/2}\|(S v_h-S_h v_h)-(S w_h-S_h w_h)\|_n$. 
By Assumption \ref{Assumption3} and the condition that $e(h)\leq C\lam_n$, we derive 
for any $w_h\in G_h$ that $\|S w_h-S_h w_h\|_n\le Ce^{1/2}(h)\|w_h\|_{X}\le Ce^{1/2}(h)\lam_n^{-1/2}\le C$.
This implies that the
diameter of $G_h$ is bounded by $C\sigma n^{-1/2}$.
Now we deduce by the maximal inequality in Lemma \ref{lem:4.2}, 
\beq\label{i1}
\|\,\sup_{w_h\in G_h}|(e,S w_h-S_h w_h)_n|\,\|_{\psi_2}\le K\int_0^{C\sigma n^{-1/2}}\sqrt{\log N\left(\frac\vep 2,G_h,\hat\sd\right)}\ d\vep.
\eeq
By Assumption \ref{Assumption3}, we know 
\ben
\hat\sd(v_h,w_h)\le C\sigma n^{-1/2}e^{1/2}(h)\|v_h-w_h\|_{X}\ \ \forall v_h,w_h\in V_h.
\een
Thus we can see 
\beq\label{i4}
\log N\left(\frac\vep 2,G_h,\hat\sd\right)=\log N\left(\frac\vep{C\sigma n^{-1/2}e^{1/2}(h)},G_h,\|\cdot\|_{X}\right).
\eeq
Now we estimate the covering entropy of $G_h$. 
First, we have $\|w_h\|_{X}\le \lam_n^{-1/2}$ for any $w_h\in G_h$. 
Noting the dimension $N_h$ of $V_h$, we obtain by Lemma \ref{lem:4.7} and \eqref{i4} that
\ben
\log N\left(\frac\vep 2,G_h,\hat\sd\right)\le CN_h(1+{\sigma n^{-1/2}e^{1/2}(h)\lam_n^{-1/2}}/\vep).
\een
Inserting this estimate in \eqref{i1}, 
\ben
\|\,\sup_{v_h\in G_h}|(e,\hat v_h-\Pi_h v_h)_n|\,\|_{\psi_2}&\le&C\int_0^{C\sigma n^{-1/2}}\sqrt{CN_h(1+\sigma n^{-1/2}e^{1/2}(h)\lam_n^{-1/2}/\vep)}\,
d\vep\\
&\le&C\sqrt{N_h}\sigma n^{-1/2}e^{1/2}(h)\lam_n^{-1/2}.
\een
This completes the proof using the condition that $N_h e(h)\leq C \lambda_n^{1-\gamma/2}$. 
\end{proof}

The following theorem presents the main results of this section.

\begin{theorem}\label{thm:4.2}
{Assume Assumption \ref{Assumption2} and \ref{Assumption3} are fulfilled.} Let $f_h\in V_h$ be the solution of \eqref{d1}. Denote by $\rho_0=\|f^*\|_{X}+\sigma n^{-1/2}$. If we take
$e(h)\leq C \lambda_n$, $N_h e(h)\leq C \lambda_n^{1-\gamma/2}$ and $\lambda_n^{1/2+\gamma/4}= O(\sigma n^{-1/2}\rho_0^{-1})$,
then there exists a constant $C>0$ such that for any $z>0$,
\ben
\mathbb{P}(\|S_h f_h-Sf^*\|_n\ge \lam_n^{1/2}\rho_0z)\le 2e^{-Cz^2} \q \mbox{and} \q 
\mathbb{P}(\|f_h\|_{X}\ge \rho_0z)\le 2e^{-Cz^2}.
\een
{More over, with the same assumption and notations in Theorem \ref{thm:2.1}, we have, 
\be\label{w-con-ds}
\mathbb{P}(\|f_h-f^*\|_{Z'}\ge \lam_n^{1/4}\rho_0z)\le 2e^{-Cz^2}.
\ee}
\end{theorem}

\begin{proof}
We first derive from \eqref{pa0} that 
\ben
\|\|f_n-f_h\|_{a_h}\|_{\psi_2}\leq C\|\inf_{v_h\in V_h} \|f_n-v_h\|_{a_h}\|_{\psi_2} + C\|\sup_{w_h\in V_h} \frac{|F(w_h)|}{\|w_h\|_{a_h}}\|_{\psi_2}.
\een
But we know $\sup_{w_h\in V_h} {|F(w_h)|}/{\|w_h\|_{a_h}}=\sup_{w_h\in G_h} |F(w_h)|$ from the proof of Theorem \ref{thm:3.1}, hence it 
suffices to estimate $\|\sup_{w_h\in G_h} |(e, S w_h-S_h w_h)_n|\|_{\psi_2}$. 
Then the desired estimates follow readily from \eqref{xx2}, Lemma \ref{lem:4.8}, the assumption that 
$\sigma n^{-1/2}=O(\lam_n^{1/2+\gamma/4}\rho_0)$ and \eqref{e3}.  

The proof of \eqref{w-con-ds} is similar to \eqref{w-con-d1}, this completes the proof.
\end{proof} 

\section{An inverse nonstationary source problem}\label{sec:heat_sourse}
In this section, we apply the theory developed in the previous section\,\ref{sec:stationary}  
to study the regularized solutions to an inverse nonstationary source problem 
associated with the heat conduction system 
\begin{equation}\label{zz0}
\left\{
\begin{aligned}
& u_t +Lu = F(x, t)\ \  \mbox{in } \Omega\times (0, T), \\
& u(x, t)= 0\ \ \mbox{on } \partial \Omega\times (0, T),\ \ \ \ u(x, 0)= 0\ \ \mbox{in } \Omega,
\end{aligned} 
\right.
\end{equation}
where $L$ is a second order elliptic operator of the form $Lu=-\nabla\cdot (a(x)\nabla u) +c(x)u$, 
and $\Omega\subset \mathbb R^d$ $(d=1,2,3)$ is a bounded domain with $C^2$ boundary 
or a convex polyhedral domain.  
We assume $a\in C^{1}(\bar{\Omega})$, $c\in C(\bar\Omega)$ with $c(x)\geq 0$ in $\Omega$, and that 
the source is of the separable form 
$F(x,t)=f(x)g(t)$ for $(x,t)\in\Om\times (0,T)$, where the temporal component $g\in H^1(0,T)$ is known and 
satisfies that $g(t)\ge 0\ ~\forall\,t\in(0,T)$, while $f(x)$ is unknown to be recovered. 

For the subsequent analysis, we first recall some standard results 
for parabolic equations (cf., e.g., \cite[\S 7.1]{Evans}). For $F\in H^1(0,T;L^2(\Om))$, we know 
the solution $u$ to \eqref{zz0} satisfies $\pa_t u\in C([0,T];L^2(\Om))\cap L^2(0,T;H^1_0(\Om))$ and 
the a priori estimate 
\ben
\|\pa_t u\|_{C([0,T];L^2(\Om))}\le C\|F\|_{H^1(0,T;L^2(\Om))}\le C\|f\|_{L^2(\Om)}.
\een
It follows then from the equation \eqref{zz0} and the regularity theory of elliptic equations that $u\in C([0,T];H^2(\Om))$ and there exists a constant $C$ such that 
\beq\label{zz-1}
\|u\|_{C([0,T];H^2(\Om))}\le C\|f\|_{L^2(\Om)}.
\eeq

Let $X=L^2(\Omega)$, $Y=H^2(\Omega)$, and 
the forward operator $S:$ $X\rightarrow Y$ be defined by $Sf=u(\cdot,T)$. 
By \eqref{zz-1} we know that $S:X\to Y$ is a bounded
operator
\ben
\|Sf\|_{H^2(\Omega)}\leq C \|f\|_{L^2(\Omega)} \q \forall\,f\in L^2(\Omega)\,.
\een

We are mainly interested in the following inverse nonstationary source problem: 

\ss\no
({\bf TIP})  Given the measurement data of $u(\cdot, t)$ at the terminal $t=T$, 
recover the spatial source distribution $f^*(x)$ in the entire domain $\Om$. 

\ss

We focus on an important physical scenario, i.e., measurement data 
is collected pointwise on a set of distributed sensors located at $\{x_i\}^n_{i=1}$ inside the domain 
$\Om$  \cite{AB05, new4, new2, new6, L08,nelson20, NNR98}.
Again, we assume the data is of the noisy form \eqref{eq:data}, 
where $\{x_i\}^n_{i=1}$ is quasi-uniformly distributed in the sense of \eqref{aa}. 

We then look for an approximate solution of the true source $f^*$ through the following least-squares regularized minimization: 
\beq\label{para-examp1}
\mathop {\rm min}\limits_{f\in X}         \|Sf-m\|^2_n+\lambda_n \|f\|_{X}^2\,.
\eeq

\subsection{Stochastic convergence for the inverse heat source problem}\label{sec:heat}
In this subsection we apply the results in section\,\ref{sec:stationary} to study the stochastic convergence of the solution 
of the problem \eqref{para-examp1} to the exact source $f^*$. 
We first recall an important property about the eigenvalue distribution for the elliptic operator $L$ \cite{agmon,Fleckinger}.

\begin{lemma}\label{para-lem:2.1} Suppose $\Omega$ is a bounded domain in $\mathbb{R}^d$ and $a, c\in C^0(\bar{\Omega})$, $c\geq 0$, then the eigenvalue problem
\beq\label{yy2}
L\psi =\mu\,\psi~ with ~ \psi_{\partial\Omega}=0
\eeq
has a countable set of positive eigenvalues $\mu_1\le\mu_2\le\cdots$,  with its corresponding eigenfunctions 
$\{\phi_k\}_{k=1}^\infty$ forming an orthogonal basis of $L^2(\Omega)$. 
Moreover, there exist constants $C_1,C_2>0$ such that 
$C_1 k^{2/d}\le \mu_k\le C_2k^{2/d}$ for all $k=1,2,\cdots.$
\end{lemma} 

With Lemma \ref{para-lem:2.1}, we can derive the important spectral property of operator $S$.

\begin{theorem}\label{para-the2.1}
Let $g\in H^1(0,T)$ and $g > 0$. Then the eigenvalue problem
\beq\label{yy3}
(\psi ,v)=\rho (S\psi ,Sv) \q \forall\, v\in X
\eeq
has a countable set of positive eigenvalues $0<\rho_1\le\rho_2\le\cdots$. Moreover, there exists a constant $C>0$ such that $\rho_k\ge Ck^{4/d}$ for all $k=1,2,\cdots$.
\end{theorem}

\begin{proof} We first consider the eigenvalue problem
\beq\label{para-exameigenS2}
\psi=\eta\, S\psi.
\eeq
Let $\{\phi_k\}_{k=1}^\infty$ be eigenfunctions of the problem \eqref{yy2} which forms an orthogonal basis of $L^2(\Omega)$. We 
write $f=\sum_{k=1}^\infty f_k \phi_k$ for a set of 
coefficients $f_k$. 
Let $u=\sum_{k=1}^\infty u_k(t) \phi_k$ be the solution of the problem \eqref{zz0}.
Plugging these two expressions of $f$ and $u$ into the first equation of \eqref{zz0}, 
we get by noting the fact that $L\phi_k=\mu_k\phi_k$ and comparing the coefficients of $\phi_k$ 
on both sides of the equation that $u_k(0)= 0$ and 
\begin{equation*}
u'_k(t) + \mu_k u_k=f_k\,g(t)\ \ \ \ \mbox{in } ~(0, T)\,.
\end{equation*}
We can write the solution as $u_k(T)=\alpha_k\,f_k$, with $\alpha_k=e^{-\mu_kT}\int_0^Te^{\mu_ks}g(s)ds$. Since $g > 0$ in $(0,T)$, we know $\al_1\ge\al_2\ge\cdots >0$. Moreover, we can easily see that $|\alpha_k|\leq 
C\mu_k^{-1}$. Noting that $Sf=u(\cdot, T)=\sum^\infty_{k=1} u_k(T) \phi_k$, we can formally write 
\ben
 S\Big(\sum_{k=1}^\infty f_k \phi_k\Big)=  \sum_{k=1}^\infty \al_k f_k \phi_k.
\een
Since $\{\phi_k\}_{k=1}^\infty$ is an orthogonal basis of $L^2(\Omega)$, we can readily see 
that the eigenvalue problem \eqref{para-exameigenS2} has a countable set of positive eigenvalues $\{\eta_k=\alpha_k^{-1}\}_{k=1}^\infty$, 
with $\{\phi_k\}_{k=1}^\infty$ being their corresponding eigenfunctions. 
By Lemma \ref{para-lem:2.1}, we have $\eta_k=\al_k^{-1}\ge C\mu_k\ge C_1k^{2/d}$. Therefore, the eigenvalue problem \eqref{yy3} has a countable set of eigenvalues $\{\rho_k\}_{k=1}^\infty$ that satisfies $\rho_k=\eta_k^2\ge Ck^{4/d}$. 
This completes the proof. 
\end{proof} 

{Next, we will certify that the abstract function space $Z$ in Theorem \ref{thm:2.1} is actually a subspace of $H^1(\Omega)$ for the inverse problem discussed in this section. So that the weaker convergence of the inverse problem corresponding to this section is $H^{-1}$ convergence under a certain assumption in the following Lemma. }
{\begin{lemma}\label{lemma-Z}
With the same notations in Lemma \ref{para-lem:2.1} and Theorem \ref{para-the2.1}. Then abstract function space $$Z=\{v\in L^2(\Omega):~ v=\sum_{k=1}^\infty v_k\phi_k,~ v_k=(v,\phi_k)_{L^2(\Omega)}~and ~ \sum_{k=1}^\infty \rho_k^{1/2}v_k^2<\infty\}$$
is actually a subspace of the Sobolev space $H^1(\Omega)$, so the dual space $H^{-1}(\Omega)\subset Z'$. Moreover, if the eigenvalues $\rho_k$ of \eqref{yy3} satisfy that $\rho_k=O( k^{4/d})$ for all $k=1,2,\cdots$, then $Z=H^1(\Omega)$ and $Z'=H^{-1}(\Omega)$.
\end{lemma}}
\begin{proof}
Since the eigenfunctions 
$\{\phi_k\}_{k=1}^\infty$ forms an orthogonal basis of $L^2(\Omega)$, then for any $v\in L^2(\Omega)$ can be expanded as 
\ben
v=\sum_{k=1}^\infty v_k\phi_k~with~ v_k=(v,\phi_k)_{L^2(\Omega)}.
\een
 
From the definition of $\{\phi_k\}_{k=1}^\infty$ in \eqref{yy2}, integrating by part, we have 
\ben
a(\phi_k,q)=\mu_k(\phi_k,q),~\forall ~ q\in H^1(\Omega),
\een 
where $a(p,q)=(ap,q)+(cp,q)$. From the ellipticity of the operator $L$ and take $q=\phi_j$, we could derive that 
\ben
\|\phi_k\|_{H^1(\Omega)}=O(\|\phi_k\|_{L^2(\Omega)}),~ (\phi_k,\phi_j)_{H^1(\Omega)}=\delta_{kj}.
\een
With the expansion $v=\sum_{k=1}^\infty v_k\phi_k$, 
\ben
\|v\|^2_{H^1(\Omega)}=O(\sum_{k=1}^\infty \mu_k v_k^2)\leq C \sum_{k=1}^\infty \rho^{1/2}_k v_k^2,
\een
this will give $\|v\|_{H^1(\Omega)}\leq C\|v\|_{Z}$. 

Moreover, if the eigenvalues $\rho_k$ satisfy that $\rho_k=O( k^{4/d})$, i.e. $\rho_k=O( \mu_k^2)$, we could derive from the above estimate that $\|v\|_{H^1(\Omega)}=O(\|v\|_{Z})$. That is to say $Z=H^1(\Omega)$ and $Z'=H^{-1}(\Omega)$.
\end{proof}
{\textbf{Remark:} In the general case, we could only conclude the eigenvalues satisfy $\rho_k\geq C k^{4/d}$ from Theorem \ref{para-the2.1}. But from the Lemma above, if we expect the space $Z=H^1(\Omega)$, we need a upper bound, i.e. $\rho_k\leq C_1 k^{4/d}$. This is actually not a strict condition, for example, we could just assume the right hand side $g(x)\geq g_{min}>0$ in Theorem \ref{para-the2.1}. With the same notations in proof of Theorem \ref{para-the2.1}, one could get
\ben
|\alpha_i|=|e^{-\mu_iT}\int_0^Te^{\mu_is}g(s)ds|\geq g_{min} |e^{-\mu_iT}\int_0^Te^{\mu_is}ds|\\
 =  g_{min} \frac{1-e^{-\mu_i T}}{\mu_i} \geq   g_{min} \frac{1}{2\mu_i}.
\een
Here one could take $T_0$ such that $e^{-\mu_1 T_0}=\frac{1}{2}$, then for $T\geq T_0$, $u_i(T)=1-e^{-\mu_i T}\geq \frac{1}{2}$.
This will readily give $\rho_k=O( k^{4/d})$. Hence $Z=H^1(\Omega)$ and $Z'=H^{-1}(\Omega)$ as conclusion of Theorem \ref{para-the2.1}. In the following section, we will always assume $\rho_k=O( k^{4/d})$, i.e. $Z=H^1(\Omega)$ and $Z'=H^{-1}(\Omega)$.}

{\bf Verification of Assumptions \ref{Assumption1} and \ref{Assumption2}.} 
We first know Assumption \ref{Assumption1}(1) holds with $\beta=4/d$ from 
\cite[Theorems 3.3-3.4]{Utreras}. 
This, along with Theorem \ref{para-the2.1}, verifies Assumption \ref{Assumption1}(2) 
with $\al=\beta=4/d$. 
Assumption \ref{Assumption2} (with $\gamma=d/2$) is a consequence of the following important estimate 
about the covering entropy \cite{Birman}.
\begin{lemma}\label{para-examlem:4.3}
Let $Q$ be the unit cube in $\R^d$ and $SW^{s,p}(Q)$ be the unit sphere of space 
$W^{s,p}(Q)$ for $s> 0$ and $p\ge 1$. Then it holds for sufficient small $\vep>0$ that 
\ben
\log N(\vep, SW^{s,p}(Q), \|\cdot\|_{L^q(Q)})\le C\vep^{-d/s},
\een
where $1\le q\le\infty$ for $sp>d$, and $1\le q\le q^*$ with $q^*=p(1-sp/d)^{-1}$
for $sp\le d$.
\end{lemma}

Under Assumptions \ref{Assumption1} and \ref{Assumption2},  the following 
two main results are direct consequences of Theorems \ref{thm:2.1} and \ref{thm:4.1}, 
respectively, for the noisy data of type {\bf (R1)} (random variables with bounded variance) 
and the noisy data of type {\bf (R2)} (sub-Gaussian random variables). 
\begin{theorem}\label{para-examthm:2.1}
For the minimizer $f_n\in L^2(\Omega)$ to the problem (\ref{para-examp1}), 
there exist constants $\lambda_0 > 0$ and $C>0$ such that the following estimates hold 
for any $\lambda_n \leq \lambda_0$:
\ben
\mathbb{E} \big[\|Sf_n-Sf^*\|^2_n\big] &\leq& C \lambda_n \|f^*\|^2_{L^2(\Omega)} + {C\sigma^2}/({n\lambda^{d/4}_n}),\\
\mathbb{E} \big[\|f_n\|^2_{L^2(\Omega)}\big] &\leq& C \|f^*\|^2_{L^2(\Omega)} + {C\sigma^2}/({n\lambda^{1+d/4}_n}).
\een
Moreover, if the eigenvalues $\rho_k$ of \eqref{yy3} satisfy that $\rho_k=O( k^{4/d})$, then
\ben
\mathbb{E} \big[\|f_n-f^*\|^2_{H^{-1}(\Omega)}\big] \leq C \lambda_n^{1/2}\|f^*\|^2_{L^2(\Omega)} + {C\sigma^2}/({n\lambda^{1/2+d/4}_n}).
\een
\end{theorem}
%
%
%
\begin{theorem}\label{para-examthm:4.1}
Let $f_n\in L^2(\Om)$ be the solution of \eqref{para-examp1} and 
$\rho_0=\|f^*\|_{L^2(\Omega)}+\sigma n^{-1/2}$. If we take $\lambda_n$ such that 
$\lambda_n^{1/2+d/8}= O(\sigma n^{-1/2}\rho_0^{-1})$,
then the following estimates hold for some constant $C>0$:
\ben
\mathbb{P}(\|Sf_n-Sf^*\|_n\ge \lam_n^{1/2}\rho_0z)\le 2e^{-Cz^2},\ \ \ \ 
\mathbb{P}(\|f_n\|_{L^2(\Om)}\ge \rho_0z)\le 2e^{-Cz^2}.
\een
Moreover, if the eigenvalues $\rho_k$ of \eqref{yy3} satisfy that $\rho_k=O( k^{4/d})$, then
\ben
\mathbb{P}(\|f_n-f^*\|_{L^2(\Om)}\ge \lambda_n^{1/4}\rho_0z)\le 2e^{-Cz^2}.
\een
\end{theorem}
%

\subsection{Finite element method for the inverse heat source problem}
In this section we consider a finite element approximation to the optimal control problem \eqref{para-examp1}
associated with the inverse heat source problem {\bf (TIP)}. 
For convenience, we assume $\Omega$ is a polygonal or polyhedral domain in $\R^d$ $(d=2,3)$. Let $\cM_h$ be a family of shape-regular and quasi-uniform finite element meshes over the domain $\Om$, and 
$V_h\subset H^1_0(\Om)$ be the conforming linear finite element space over the mesh $\cM_h$. 
We divide the time
interval $(0,T)$ into a uniform grid with time step size $\tau=T/N$ and write $t^i=i\tau$ for $i=0,1,...,N$. 

We will use the backward Euler scheme in time and the linear finite element method in space to approximate the heat conduction problem \eqref{zz0}: 
Find $u_h^i\in V_h$, $i=1,2,\cdots,N$, such that
\beq\label{zz1}
\Big(\frac{u_h^i-u_h^{i-1}}\tau,v_h\Big) + a(u^i_h,v_h)=(fg^{i},v_h)\ \ \ \ \forall v_h\in V_h,
\eeq
where $a(v,w)=(a\nabla v,\nabla w) +(cv,w)$ for any $v,w\in H^1_0(\Om)$. We approximate the forward solution $Sf$ by $S_{\tau,h}f=u^N_h$. 
The inverse problem \eqref{para-examp1} can be approximated by the following least-squares problem
\beq\label{para-inversedis}
\mathop {\rm min}\limits_{f\in V_h}         \|S_{\tau,h}f-m\|^2_n+\lambda_n \|f\|_{L^2(\Omega)}^2.
\eeq
We shall make use of the results in section\,\ref{sec:heat} to study the stochastic convergence of the solution $f_{\tau,h}$ of
the problem \eqref{para-inversedis} to the true solution $f^*\in L^2(\Om)$.

\ss
{\bf Verification of Assumption \ref{Assumption3}.}
Let $P_h: L^2(\Om)\to V_h$ be the orthogonal projection operator in the $L^2$ inner product. For any $f\in X=L^2(\Om)$, we know
from \eqref{zz1} that $S_{\tau,h}f=S_{\tau,h}(P_hf)$. Therefore, Assumption \ref{Assumption3}\,(2) is trivially satisfied. It remains to check
Assumption \ref{Assumption3}\,(1), which amounts to derive the error estimate of the fully discrete method \eqref{zz1}. The
classical theory for the implicit Euler scheme in time and finite element method in space for solving parabolic equations requires the regularity $\pa_{tt}u\in L^1(0,T;L^2(\Om))$ of the solution of the problem \eqref{zz0} (see e.g., \cite[Chapter 1]{Thomee}). This regularity requires the compatibility condition $F(x, 0)=f(x)g(0)=0$ on $\pa\Om$, 
which may not be convenient to meet in practice. 
Instead, we will derive an error estimate in the remaining part of this section, 
without this compatibility condition,  
by adapting some arguments in \cite[Chapter 3]{Thomee} 
for the error estimates of finite element solutions to parabolic equations with rough initial data. 

We start with the weak $W^{2,1}(0,T;L^2(\Om))$ regularity for the solution to \eqref{zz0}.

\begin{lemma}\label{lem:new1}
Let $F(x,t)=f(x)g(t)$ for $(x,t)\in\Om\times(0,T)$, with $g\in H^2(0,T)$. 
Then there exists a generic constant $C$ such that the solution $u$ to \eqref{zz0} satisfies
\ben
& &\|\pa_t u\|_{C([0,T];L^2(\Om))}\le C\|F(\cdot,0)\|_{L^2(\Om)}+C\int^T_0\|\pa_t F\|_{L^2(\Om)}dt,\\
& &\|t\pa_{tt}u\|_{C([0,T];L^2(\Om))}\le C\|F(\cdot,0)\|_{L^2(\Om)}+C\int^T_0(\|\pa_tF\|_{L^2(\Om)}+t\|\pa_{tt} F\|_{L^2(\Om)})dt\,.
\een
\end{lemma}

\begin{proof} The proof follows from the standard energy argument, so only an outline is given here. 
We differentiate the first equation in \eqref{zz0} in time to see that $v(x,t)=\pa_t u$ satisfies the conditions 
that $v=0$ on $\pa\Om\times (0,T)$ and $v(x,0)=F(x,0)$ in $\Om$, and 
\beq\label{zz2}
\pa_t v+Lv=\pa_{t}F(x,t)\ \ \mbox{in }\Om\times (0,T).
\eeq
Then the first estimate in the lemma follows 
by multiplying both sides of equation \eqref{zz2} by $v$ and integrating by parts.

Next we multiply both sides of \eqref{zz2} by $t\pa_t  v$, then integrate by parts and apply 
the first estimate in the lemma to get 
\beq\label{zz4}
\int^t_0 t\|\pa_t v\|_{L^2(\Om)}^2dt\le C\|F(\cdot,0)\|_{L^2(\Om)}^2+C\Big(\int^T_0\|\pa_t F\|_{L^2(\Om)}dt\Big)^2+C\int^T_0t\|\pa_t F\|_{L^2(\Om)}^2dt.
\eeq
Finally, we differentiate the equation \eqref{zz2} in time to get
\ben
\pa_{tt}v+L(\pa_t v)=\pa_{tt}F(x,t)\ \ \ \ \mbox{in }\Om\times (0,T).
\een
By multiplying both sides of the equation by $t^2\pa_t v$, integrating by parts again and applying \eqref{zz4}, 
we obtain
\ben
t\|\pa_tv\|_{L^2(\Om)}\le C\|F(\cdot,0)\|_{L^2(\Om)}&+&C\int^T_0(\|\pa_t F\|_{L^2(\Om)}+t\|\pa_{tt}F\|_{L^2(\Om)})dt\\
&+&C\Big(\int^T_0t\|\pa_t F\|_{L^2(\Om)}^2dt\Big)^{1/2}, 
\een
which implies the second estimate of the lemma by noticing that
\ben
\int^T_0t\|\pa_t F\|_{L^2(\Om)}^2dt&\le&\sup_{t\in (0,T)}\|t\pa_t F\|_{L^2(\Om)}\cdot\int^T_0\|\pa_t F\|_{L^2(\Om)}dt\\
&=&\sup_{t\in (0,T)}\left\|\int^t_0\pa_s(s\pa_s F(s))ds\right\|_{L^2(\Om)}\cdot\int^T_0\|\pa_t F\|_{L^2(\Om)}dt\\
&\le&\int^T_0\left(\|\pa_t F\|_{L^2(\Om)}+t\|\pa_{tt}F\|_{L^2(\Om)}\right)dt\cdot\int^T_0\|\pa_t F\|_{L^2(\Om)}dt.
\een 
This completes the proof. 
\end{proof}

\begin{lemma}\label{lem:new2}
Let $u_h\in H^1(0,T;V_h)$ be the following semi-discrete finite element solution of the problem \eqref{zz0}:
\beq\label{zz5}
(\pa_t u_h,v_h)+a(u_h,v_h)=(F,v_h)\ \ \ \ \forall v_h\in V_h\ \ \mbox{a.e. in }(0,T).
\eeq
Then there exists a constant $C$ independent of the mesh size $h$ such that
\ben
\|u-u_h\|_{C([0,T];L^2(\Om))}\le Ch^2\max_{t\in [0,T]}(\|\pa_t u\|_{L^2(\Om)}+\|t\pa_{tt}u\|_{L^2(\Om)}+\|F\|_{L^2(\Om)}+\|t\pa_{t}F\|_{L^2(\Om)}),
\een
where $h=\max_{K\in\cM}h_K$ and $h_K$ is the diameter of the element $K\in\cM$.
\end{lemma}

\begin{proof} We follow the argument in \cite[Chapter 3]{Thomee}. 
Define $G:L^2(\Om)\to H^1_0(\Om)$ and $G_h: L^2(\Om)\to V_h$ such that for any $w\in L^2(\Om)$, $Gw\in H^1_0(\Om), G_hw\in V_h$ sastify
\ben
a(Gw,v)=(w,v)\ \ \forall v\in H^1_0(\Om); \quad a(G_hw,v_h)=(w,v_h)\ \ \forall v_h\in V_h.
\een
The equations \eqref{zz0} and \eqref{zz5} can be reformulated as
\ben
\pa_t(Gu)+u=GF,\ \ \pa_t(G_hu_h)+u_h=G_hF.
\een
Writing $e=u-u_h$, then we know $e$ satisfies
\ben
G_h(\pa_te)+e=\rho\ \ \mbox{a.e. in }(0,T),\ \ \ \ (G_he)(\cdot,0)=0\ \ \mbox{in }\Om,
\een 
where $\rho=(G_h-G)(\pa_t u)+(G-G_h)F$. By the argument in the proof of Lemma \ref{lem:new1} we can obtain 
(see \cite[Lemma 3.4]{Thomee}) that
\ben
\max_{t\in [0,T]}\|e\|_{L^2(\Om)}\le C\max_{t\in [0,T]}(\|\rho(t)\|_{L^2(\Om)}+\|t\pa_t\rho(t)\|_{L^2(\Om)}).
\een
This completes the proof by noting that $\|Gw-G_h w\|_{L^2(\Om)}\le Ch^2\|w\|_{L^2(\Om)}$ \,$\forall w\in L^2(\Om)$, 
which follows by the Aubin-Nitsche argument since the domain $\Om$ is convex. 
\end{proof}

The following lemma for the error estimate of the fully discrete finite element method was not covered 
by the general results in \cite[Chaper 8]{Thomee} since we do not have the condition that 
$F(x,0)=0$ on $\partial \Om$ here, which was critical in \cite{Thomee}.

\begin{lemma}\label{lem:new3}
Let $u_h\in H^1(0,T;V_h)$ be the solution of the problem \eqref{zz5} and $u_h^i\in V_h, i=1,2,\cdots,N$, be 
the solution of the problem \eqref{zz1}. Then there exists a constant $C$ independent of $h,\tau$ such that
\ben
\max_{1\le i\le N}\|u_h(\cdot,t_i)-u_h^i\|_{L^2(\Om)}\le C\tau(1+\ln N)(\|F\|_{C([0,T];L^2(\Om))}+\|\pa_t F\|_{C([0,T];L^2(\Om))}).
\een
\end{lemma}

\begin{proof} Let $\{\lam_j\}_{j=1}^M$ be the eigenvalues of the eigenvalue problem
\ben
a(\phi_h,v_h)=\lam(\phi_h,v_h)\ \ \ \ \forall v_h\in V_h,
\een
and $\{\phi_j\}^M_{i=1}$ be the corresponding eigenfunctions which form an orthonormal basis of $V_h$ in the $L^2(\Om)$ norm.
By the Poincar\'e inequality, we know that $\lam_j\ge C$, $j=1,2,\cdots,M$, for some constant $C$ independent of the mesh size $h$.

We write $u_h(x,t)=\sum^M_{j=1}u_j(t)\phi_j(x)$ and $F(x,t)=\sum^M_{j=1}F_j(t)\phi_j(x)$, where $u_j(t)=(u_h(\cdot,t),\phi_j)$ and
$F_j(t)=(F(\cdot,t),\phi_j)$. Then it follows from \eqref{zz5} that
\ben
u_j'(t)+\lam_j u_j=F_j(t)\ \ \ \ \mbox{a.e. in }(0,T), 
\een
whose solution can be written as 
\beq\label{cc1}
u_j(t^i)=\int^{t^i}_0e^{\lam_j(s-t^i)}F_j(s)ds=\int^{t^i}_0e^{-\lam_jt}F_j(t^i-t)dt.
\eeq
Similarly, we write $u_h^i=\sum^M_{j=1}U^i_j\phi_j$, where $U^i_j=(u^i_h,\phi_j)$, $i=1,2,\cdots,N, j=1,2,\cdots,M$. From \eqref{zz1} we know that
\ben
\frac 1\tau(U^i_j-U^{i-1}_j)+\lam_jU^i_j=F^i_j:=F_j(t^i),\ i=1,2,\cdots,N,j=1,2,\cdots,M.
\een
This implies that $U^i_j=r(\lam_j)U^{i-1}_j+\tau r(\lam_j\tau)F^i_j$, where $r(t)=(1+t)^{-1}\ \forall t\ge 0$, 
hence 
\beq\label{cc2}
U^i_j=\sum^i_{k=1}\tau r(\lam_j\tau)^k F_j^{i-k+1}.
\eeq
For any $j=1, \cdots,M$, we distinguish two cases. If $\lam_j\tau\ge 1$, we know from \eqref{cc1} that 
\ben
|u_j(t^i)|\le\|F_j\|_{C[0,T]}\int^{t^i}_0e^{-\lam_j t}dt=\lam_j^{-1}(1-e^{-\lam_jt^i})\|F_j\|_{C[0,T]}\le \tau\|F_j\|_{C[0,T]}.
\een
On the other hand, we obtain from \eqref{cc2} that 
\ben
|U^i_j|\le\Big(\sum^i_{k=1}2^{-k}\Big)\tau\|F_j\|_{C[0,T]}\le 2\tau\|F_j\|_{C[0,T]}.
\een
Therefore, we derive for $\lam_j\tau\ge 1$ that 
\beq\label{cc3}
|u^i_j(t^i)-U^i_j|\le C\tau\|F_j\|_{C[0,T]}.
\eeq

Now we consider the case when $\lam_j\tau\le 1$. By \eqref{cc1} we have 
\ben
u_j(t^i)&=&\sum^i_{k=1}\int^{t^k}_{t^{k-1}}e^{-\lam_j t}(F_j(t^i-t)-F(t^i-t^{k-1}))dt+\sum^i_{k=1}\int^{t^k}_{t^{k-1}}e^{-\lam_j t} F_j^{i-k+1}dt\\
&=&\sum^i_{k=1}\int^{t^k}_{t^{k-1}}e^{-\lam_j t}(F_j(t^i-t)-F(t^i-t^{k-1}))dt+\sum^i_{k=1}\tau\,\frac{e^{\lam_j\tau}-1}{\lam_j\tau}\,e^{-k\lam_j\tau}F^{i-k+1}_j,
\een
which, together with \eqref{cc2}, yields
\be\label{cc4}
u_j(t^i)-U^i_j&=&\sum^i_{k=1}\tau\left(\frac{e^{\lam_j\tau}-1}{\lam_j\tau}e^{-k\lam_j\tau}-r(\lam_j\tau)^k\right)F_j^{i-k+1}\nn\\
& &\ +\,\sum^i_{k=1}\int^{t^k}_{t^{k-1}}e^{-\lam_j t}(F_j(t^i-t)-F(t^i-t^{k-1}))dt
:={\rm I}+{\rm II}.
\ee
Recalling the following elementary estimate in \cite[(7.22)]{Thomee}, 
\ben
|e^{-kt}-r(t)^k|\le Ck^{-1}\ \ \ \ \forall \,t\ge 0, \ \forall \,k=1,2,\cdots,
\een
and using the fact that $({t^{-1}(e^t-1)}-1)/(1-e^{-t})$ is bounded for $0\le t\le 1$, 
we obtain
\ben
|{\rm I}|&\le&\sum^i_{k=1}\tau\left|\Big(\frac{e^{\lam_j\tau}-1}{\lam_j\tau}-1\Big)e^{-k\lam_j\tau}+(e^{-k\lam_j\tau}-r(\lam_j\tau)^k)\right|
|F_j^{i-k+1}|\\
&\le&C\tau\left[\left(\frac{e^{\lam_j\tau}-1}{\lam_j\tau}-1\right)\frac{1}{1-e^{-\lam_j\tau}}+\sum^i_{k=1}k^{-1}\,\right]\|F_j\|_{C[0,T]}\\
&\le&C(1+\ln i)\tau\|F_j\|_{C[0,T]}\,.
\een
The term ${\rm II}$ can be bounded by the standard argument as follows:
\ben
{\rm II}\le C\tau\|\pa_t F_j\|_{C[0,T]}\int^{t^i}_0e^{-\lam_j t}dt\le C\lam_j^{-1}\tau\|\pa_t F_j\|_{C[0,T]}\le C\tau\|\pa_t F_j\|_{C[0,T]},
\een
where we have used the fact that $\lam_j\ge C$ for some constant $C$ independent of $h$. 

Combining \eqref{cc3}, \eqref{cc4} and the above two estimates we obtain
\ben
|u_j(t^i)-U^i_j|\le C\tau(1+\ln N)(\|F_j\|_{C[0,T]}+\|\pa_t F_j\|_{C[0,T]}).
\een
This completes the proof. 
\end{proof}

By Lemmata \ref{lem:new1}-\ref{lem:new3}, we know that under the condition $g\in H^2(0,T)$, 
\beq\label{cc5}
\|S_{\tau,h}f-Sf\|_{L^2(\Om)}\le C(h^2+\tau |\ln \tau|)\|f\|_{L^2(\Om)},
\eeq
for some constant $C$ which depends possibly on $T$, $\|g\|_{H^2(0,T)}$ but is independent of $h$ and $\tau$. 

Assumption \ref{Assumption3}\,(1) is now a consequence of the following lemma. 

\begin{lemma}\label{examlem:discrete}
If $g\in H^2(0,T)$, $S_{\tau,h}f=u^N_h$ with $u_h^N$ being the solution of the problem \eqref{zz1}, then for any $f\in L^2(\Omega)$, there exists a constant $C$ independent of $h$ and $\tau$ such that
\ben
\|Sf-S_{\tau,h}f\|_n\leq C(h^2+\tau|\ln \tau|)\|f\|_{L^2(\Omega)}.
\een
\end{lemma}

\begin{proof}
Let $\Pi_h:C(\bar\Om)\to V_h$ be the canonical finite element interpolant, 
then we know from the standard interpolation theory of finite element methods \cite{Ciarlet} that
\ben
& &\|Sf-\Pi_h(Sf)\|_{L^\infty(K)}\leq Ch^{2-d/2}\|Sf\|_{H^2(K)} \quad \forall K\in \cM_h,\\
& &\|Sf-\Pi_h(Sf)\|_{L^2(K)}\leq Ch^{2}\|Sf\|_{H^2(K)} \quad \forall K\in \cM_h.
\een
Let $\mathbb{T}_{K}=\{x_i: x_i\in K, 1\le i\le n\}$. By the assumption that $\{x_i\}^n_{i=1}$ is 
quasi-uniformly distributed and the mesh $\cM_h$ is quasi-uniform, we know that the cardinal $\#\mathbb{T}_{K}\le Cnh^d$. Thus we have
\ben
\|Sf-\Pi_h(Sf)\|_n^2\le\frac 1n\sum_{{K}\in\cM_h}\#\mathbb{T}_{K}\|Sf-\Pi_h(Sf)\|_{L^\infty({K})}^2\le Ch^4\|Sf\|_{H^2(\Om)}^2.
\een
On the other hand, we can derive by making use of inverse estimates that 
\ben
\|S_{\tau,h} f-\Pi_h(Sf)\|_n^2 & \le&\frac 1n\sum_{{K}\in\cM_h}\#\mathbb{T}_{K}\|S_{\tau,h} f-\Pi_h(Sf)\|_{L^\infty({K})}^2\\
 & \le&\frac 1n\sum_{{K}\in\cM_h}\#\mathbb{T}_{K}|{K}|^{-1}\|S_{\tau,h} f-\Pi_h(Sf)\|_{L^2({K})}^2\\
&\le&C\|S_{\tau,h} f-\Pi_h(Sf)\|_{L^2(\Om)}^2\\
&\le&C\|S_{\tau,h} f- Sf\|_{L^2(\Om)}^2+ C\|\Pi_h(Sf)- Sf\|_{L^2(\Om)}^2\\
&\leq&C\|S_{\tau,h} f- Sf\|_{L^2(\Om)}^2+Ch^4\|Sf\|_{H^2(\Om)}^2.
\een
Therefore,
\ben
\|Sf-S_{\tau,h} f\|_n\le C\|S_{\tau,h} f- Sf\|_{L^2(\Om)}+Ch^2\|f\|_{L^2(\Om)}.
\een
This completes the proof by \eqref{cc5}. 
\end{proof}

After the verification of Assumption \ref{Assumption3},  
the following stochastic convergence of the finite element method to the inverse heat source problem 
follows readily from Theorem \ref{thm:3.1}. 

\begin{theorem}\label{examthm:3.1}
Let $g\in H^2(0,T)$ and the measurement data \eqref{eq:data} be of the type {\bf (R1)}. 
Then there exist constants $\lambda_0 > 0$ and $C>0$ such that for any $\lambda_n \leq \lambda_0$ 
and $\tau|\ln\tau|=O(h^2)$, the following estimates hold for the solutions $f_n\in L^2(\Om)$ 
to \eqref{para-examp1} and $f_h\in V_h$ to \eqref{para-inversedis}:
\ben
& &\mathbb{E}\big[\|Sf^*- S_{\tau,h}f_h\|_n^2\big]\le C(\lam_n+h^4)\|f^*\|^2_{L^2(\Om)}+C\left(1+\frac{h^4}{\lam_n}\right)\frac{\sigma^2}{n\lambda_n^{d/4}},\\
& &\mathbb{E}\big[\|f^*- f_h\|_{L^2(\Omega)}^2\big]\le C\Big(1+\frac{h^4}{\lam_n}\Big)\|f^*\|^2_{L^2(\Om)}+
C\left(1+\frac{h^4}{\lam_n}\right)\frac{\sigma^2}{n\lambda_n^{1+d/4}}.
\een
{\ben
\mathbb{E}\big[\|f^*- f_h\|_{H^{-1}(\Omega)}^2\big]\le C(\lam^{1/2}_n+h^2)\Big(1+\frac{h^4}{\lam_n}\Big)\|f^*\|^2_{L^2(\Om)}+
C(\lam^{1/2}_n+h^2)\left(1+\frac{h^4}{\lam_n}\right)\frac{\sigma^2}{n\lambda_n^{1+d/4}}.
\een}
\end{theorem}

\begin{proof}
Since the mesh is assumed to be quasi-uniform,
the dimension $N_h$ of the linear finite element space $V_h$ is bounded by $N_h\leq Ch^{-d}$. By Theorem \ref{para-the2.1}, we know that $\al=4/d$. Take
$\tau|\ln\tau|=O(h^2)$, then we know from Theorem \ref{thm:3.1} that 
\ben
& &\mathbb{E}\big[\|Sf^*- S_{\tau,h}f_h\|_n^2\big]\le C(\lam_n+h^4)\|f^*\|^2_{L^2(\Om)}+C\left[1+\frac{h^4}{\lam_n}+\Big(\frac{h^4}{\lam_n}\Big)^{1-\frac d4}\right]\frac{\sigma^2}{n\lambda_n^{1+d/4}}.
\een
We can easily check that $({h^4}/{\lam_n})^{1-\frac d4}\leq 1$ for ${h^4}/{\lam_n}\leq 1$, and 
$({h^4}/{\lam_n})^{1-\frac d4}\leq {h^4}/{\lam_n}$ for ${h^4}/{\lam_n}\geq 1$. 
Therefore, we have $({h^4}/{\lam_n})^{1-\frac d4}\leq 1+{h^4}/{\lam_n}$. 
This leads to the conclusions of Theorem \ref{examthm:3.1}.
\end{proof}

We end this section with the following convergence of the finite element method 
to the inverse heat source problem {\bf (TIP)}, directly following from Theorem 
\ref{thm:4.2} by noticing that $N_h\le Ch^{-d}\le C\lam_n^{-\gamma/2}$ with $\gamma=d/2$. 

\begin{theorem}\label{examthm:4.2}
Let $g\in H^2(0,T)$, the measurement data \eqref{eq:data} is of the type {\bf (R2)} and 
$\rho_0=\|f^*\|_{L^2(\Omega)}+\sigma n^{-1/2}$. If we take
$h=O(\lam_n^{1/4}), \tau|\ln\tau|=O(\lam_n^{1/2})$, and $\lambda_n^{1/2+d/8}= O(\sigma n^{-1/2}\rho_0^{-1})$,
then there exists a constant $C>0$ such that for any $z>0$,
\ben
\mathbb{P}(\|S_{\tau,h}f_h-Sf^*\|_n\ge \lam_n^{1/2}\rho_0z)\le 2e^{-Cz^2},\ \ \ \ 
\mathbb{P}(\|f_h\|_{L^2(\Omega)}\ge \rho_0z)\le 2e^{-Cz^2}.
\een
{\ben
\mathbb{P}(\|f_h-f^*\|_{H^{-1}(\Omega)}\ge \lambda_n^{1/4}\rho_0z)\le 2e^{-Cz^2}.
\een}
\end{theorem}

\section{Numerical examples}\label{sec:numerics}
In this section, we present several numerical examples to confirm the theoretical results 
in previous sections. 
We take the domain $\Omega=(0,1)\times(0,1)$ and a set of uniformly distributed measurement 
locations $\{x_i\}_{i=1}^n$ in $\Omega$. 
In all examples below, we take the coefficients $a(x)=1, c(x)=0$,
which fullfil the uniform ellipticity condition, and $g(t)\equiv 1$, $T=1$. 
The finite element mesh ${\cal M}_h$ of $\Om$ is constructed by first dividing $\Om$ into $h^{-1}\times h^{-1}$ uniform rectangles and then connecting the lower left and upper right vertices of each rectangle. 
We set the noise $e_1$, $\cdots$, $e_n$ in the dataset \eqref{eq:data} to be the normal random variables with variance $\sigma$. 

Motivated by Theorem \ref{para-examthm:2.1}, 
we propose a self-consistent algorithm to determine the regularization parameter $\lam_n$ in 
\eqref{para-inversedis} based on the rule
\beq\label{para-examo1}
\lam_n^{1/2+d/8}=\sigma n^{-1/2}\|f^*\|_{L^2(\Om)}^{-1}. 
\eeq 
This choice requires the knowledge of the true source function $f^*$ and the noise level $\sigma$. 
We propose now a self-consistent algorithm to determine the parameter $\lam_n$, 
without knowing the true source function $f^*$ and the noise level $\sigma$. 
To do so, we estimate $\|f^*\|_{L^2(\Om)}$ by $\|f_h\|_{L^2(\Om}$ and $\sigma$ by $\|S_{\tau,h}f_h-m\|_n$ 
since $\|Sf^*-m\|_n=\|e\|_n$. This is expected to yield 
a good estimate of the variance by the law of large numbers.

 \begin{alg}[Computing an estimate of the regularization parameter $\lam_n$]
 \label{para-examj1}
 
 $1^\circ$ Given an initial guess of $\lambda_{n,0}$; for $j=0, 1, \cdots$, do the following 
 
 $2^\circ$ Solve \eqref{para-inversedis} for $f_h$ 
 with $\lambda_{n}$ replaced by $\lambda_{n,j}$ over the mesh $\cM_h$;
 
 $3^\circ$ Update $\lambda_{n,j+1}$: 
 ~~$\lambda^{1/2+d/8}_{n,j+1} = n^{-1/2} \|S_{\tau,h} f_h-m\|_n \,\|f_h\|^{-1}_{L^2(\Om)}$.
\end{alg}    

A natural choice of the initial guess is $\lam_{n,0}=n^{- 4/{(d+4)}}$ since $f^*$ and $\sigma$ are unknown, which is used in our numerical examples. 

\begin{example} \label{para-examnumerical.1} 
This example is used to verify the nearly optimality of 
the choice of the smoothing parameter $\lambda_n$ suggested by (\ref{para-examo1}). 
We choose $n=10^4$, 
$\sigma=0.1$ or $\sigma=0.01$, 
and the mesh size $h=0.05$ and the time step size $\tau=0.01$, which 
are sufficiently small so that the finite element errors are negligible. 
We take the true source $f^*$ to be the function whose surface is given as in 
Fig.\,\ref{para-examexact}.  
\end{example}

Example \ref{para-examnumerical.1} demonstrates the nearly optimality of 
the choice of the smoothing parameter $\lambda_n$ suggested by (\ref{para-examo1}). 
In fact, 
we have $\|f^*\|_{L^2(\Omega)}\approx 0.54$, then \eqref{para-examo1} suggests 
$\lambda_n\approx 2.3 \times 10^{-4}$ (for $\sigma=0.1$) 
and $\lambda_n\approx 1.1 \times 10^{-5}$ (for $\sigma=0.01$). 
These two approximate $\lambda_n$'s are indeed very close 
to the optimal $\lambda_n=1 \times 10^{-4}$ (for $\sigma=0.1$)  and $\lambda_n=1 \times 10^{-5}$
(for $\sigma=0.01$), which we have estimated by computing the errors 
$\|Sf^*-S_{\tau,h}f_h\|_n$ with 10 different choices of regularization parameter: 
$\lam_{n,k}=10^{-k}$ ($k=1,2,\cdots,10$), see Fig.\,\ref{para-examexample.1}.

\begin{figure}[t]
\begin{center}
\begin{tabular}{cc}
\includegraphics[width=6cm]{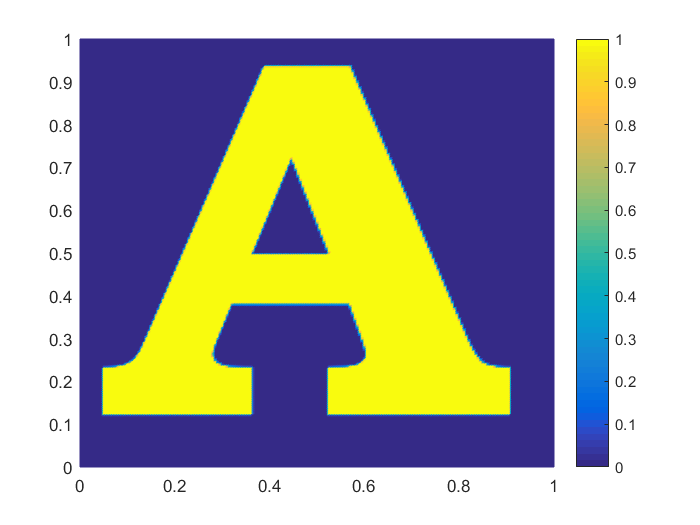} 
\end{tabular}
\end{center}
\caption{\em The surface plot of the exact solution $f^*$.}
\label{para-examexact}
\end{figure}

\begin{figure}[t]
\begin{center}
\begin{tabular}{cc}
\includegraphics[width=5cm]{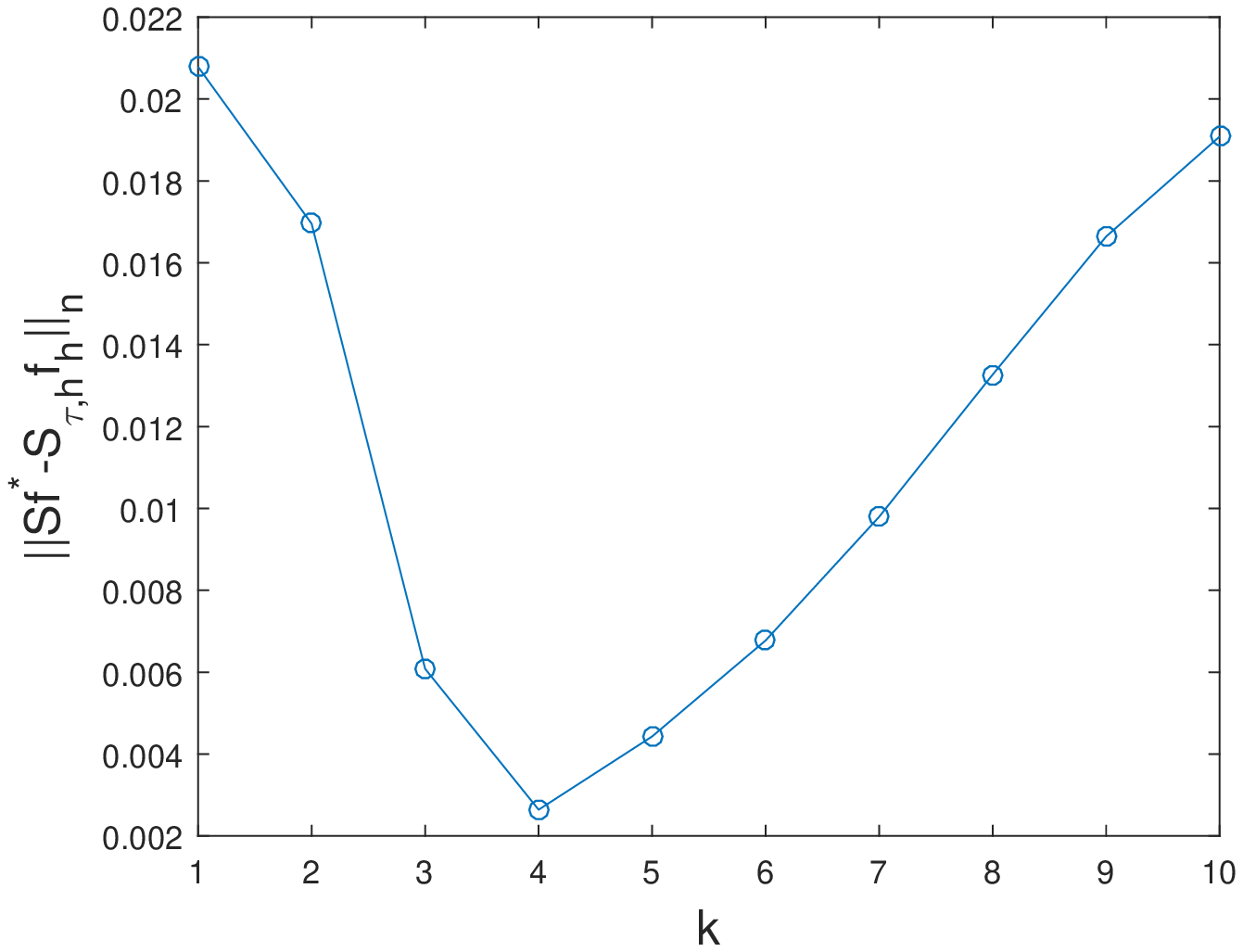} 
\includegraphics[width=5cm]{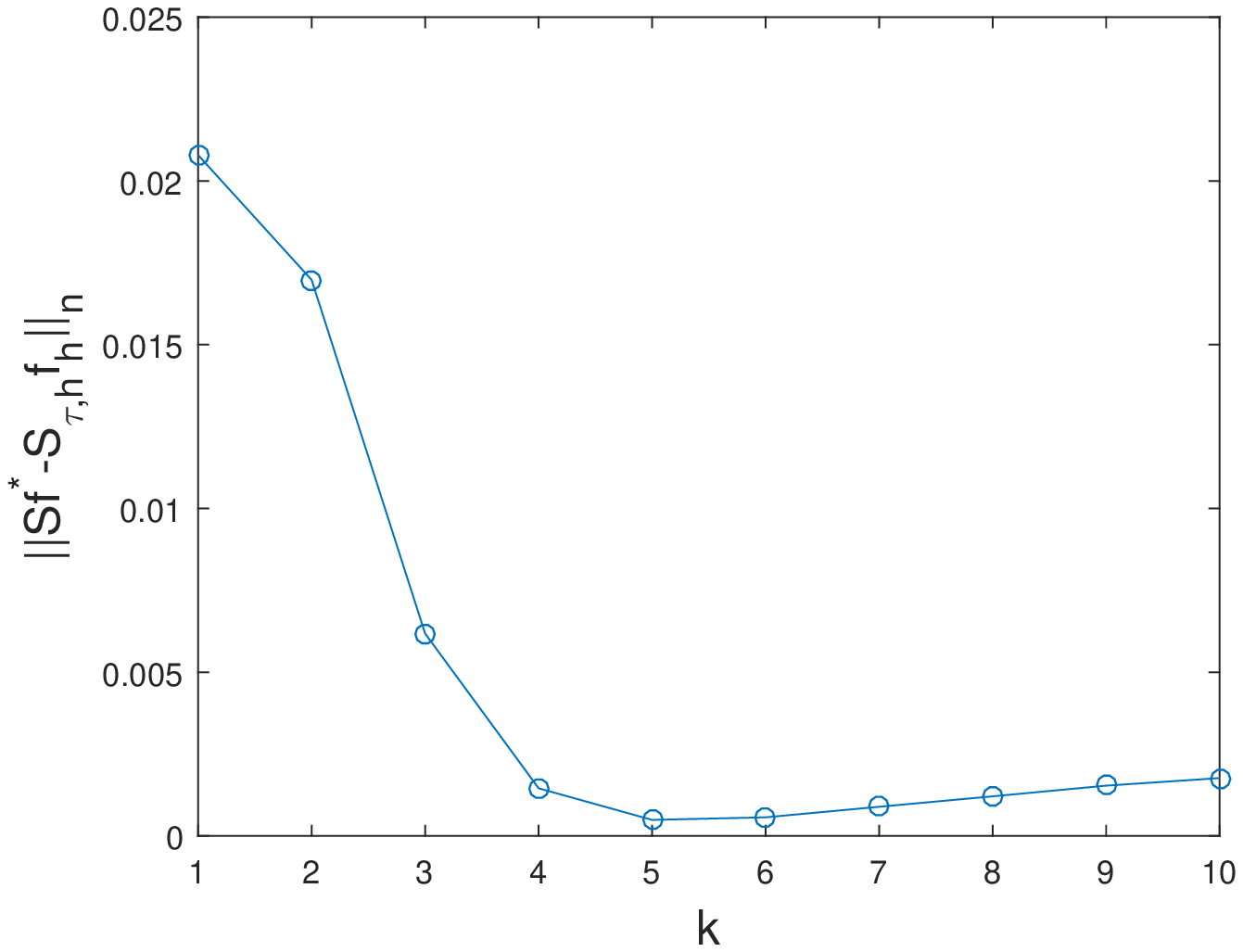} 
\end{tabular}
\end{center}
\caption{\em The empirical errors $\|Sf^*-S_{\tau,h}f_h\|_n$ with $\lambda_n=10^{-k}$ $(k=1,\cdots,10)$ 
for $\sigma=0.1$ (left) and $\sigma=0.01$ (right).}
\label{para-examexample.1}
\end{figure}




\begin{example} \label{para-examnumerical.32}
This example is presented to verify if the probability density function of the empirical error 
$\|Sf^*-S_{\tau,h} f_h\|_n$ has an exponentially decaying tail. 
We set the variance $\sigma=0.001$, $n=25\times 10^4$, and 
choose the mesh size $h$ and time step size $\tau$ to be
small enough so that the finite element errors are negligible.
We take 10,000 samples and compute the empirical error $\|Sf^*-S_{\tau,h} f_h\|_n$ for each sampling. 
\end{example}

In Example \ref{para-examnumerical.32}, we can compute that $\|Sf^*\|_{L^{\infty}(\Omega)}\approx 0.04$, 
so the relative noise level $\sigma/\|Sf^*\|_{L^\infty(\Om)}$ is about $2.5\%$ for this example. 
Figure \ref{para-examexample.32}(a) shows the histogram plot of the empirical errors, 
while Figure \ref{para-examexample.32}(b) plots the quantile-quantile (Q-Q) plot to compare the sample distribution of the empirical error with the standard normal distribution. The Q-Q plot is a standard graphic tool in statistics to check the data distribution \cite{G}. If the sample distribution is indeed normal, the Q-Q plot 
should give a scattered plot, where the points show a linear relationship between the sample and the theoretical quantiles. We can observe from Figure \ref{para-examexample.32} (right)  that almost all the points are concentrated around the dotted line, which implies that the overall distribution of the error is very close to 
a normal distribution. Moreover, the points around the two ends are also not far from the line, which indicates that the tail distribution of the error is also close to a Gaussian tail, as indicated in 
Theorem\,\ref{examthm:4.2}. 
The probability density function is computed by the Matlab function 'qqplot'. 

\begin{figure}[t]
\begin{center}
\begin{tabular}{cc}
\includegraphics[width=6cm]{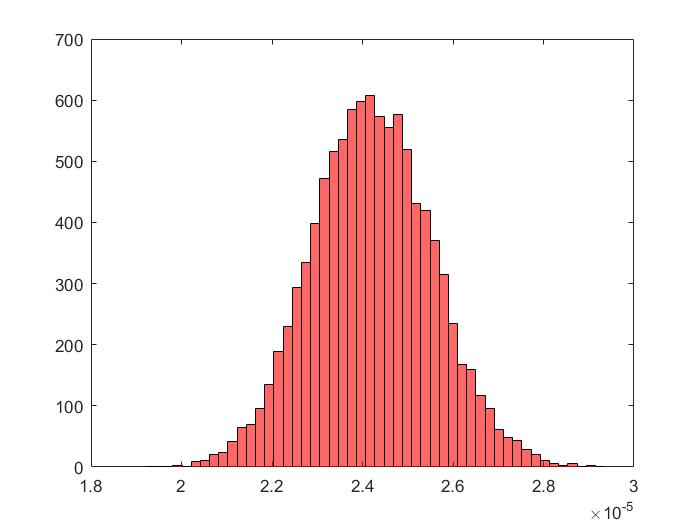} & 
\includegraphics[width=6cm]{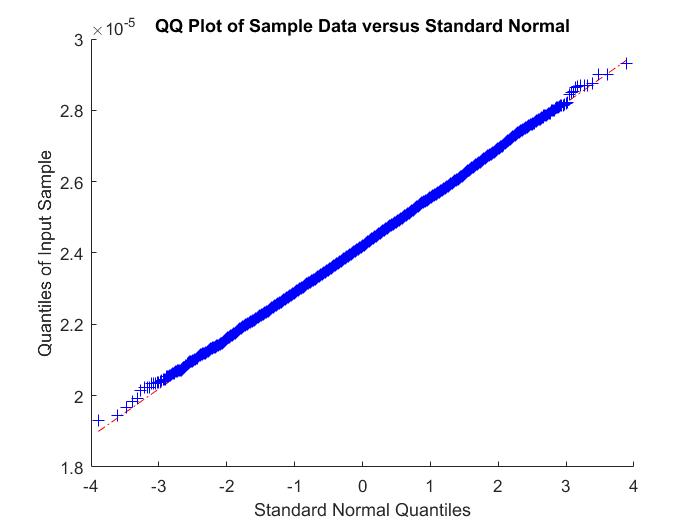} 
\end{tabular}
\end{center}
\caption{\em The histogram (left) and quantile-quantile (right) 
plots of the empirical error $\|S_{\tau,h} f_h-Sf^*\|_n$ with $10,000$ samples.}\label{para-examexample.32}
\end{figure}

\begin{example}\label{para-examnumerical.5}
This example is to confirm Theorems \ref{examthm:3.1} and \ref{examthm:4.2}, namely, 
to verify if the empirical error $\|Sf^*-S_{\tau,h} f_h\|_n$ depends linearly on $\lambda_n^{1/2}$ 
when the regularization parameter $\lambda_n$ is taken by the optimal choice \eqref{para-examo1}. 
The mesh size $h=\lambda_n^{1/4}$ and the time step size $\tau|\ln\tau|=\lambda_n^{1/2}$ are chosen according to Theorems \ref{examthm:3.1} and \ref{examthm:4.2}. 
We take the true source $f^*$ to be the function given in Figure \ref{para-examexact}, 
and $n$ to change from $25\times 10^2$ to $25\times 10^4$. 
\end{example}

\begin{figure}[t]
\begin{center}
\begin{tabular}{cc}
\includegraphics[width=6cm]{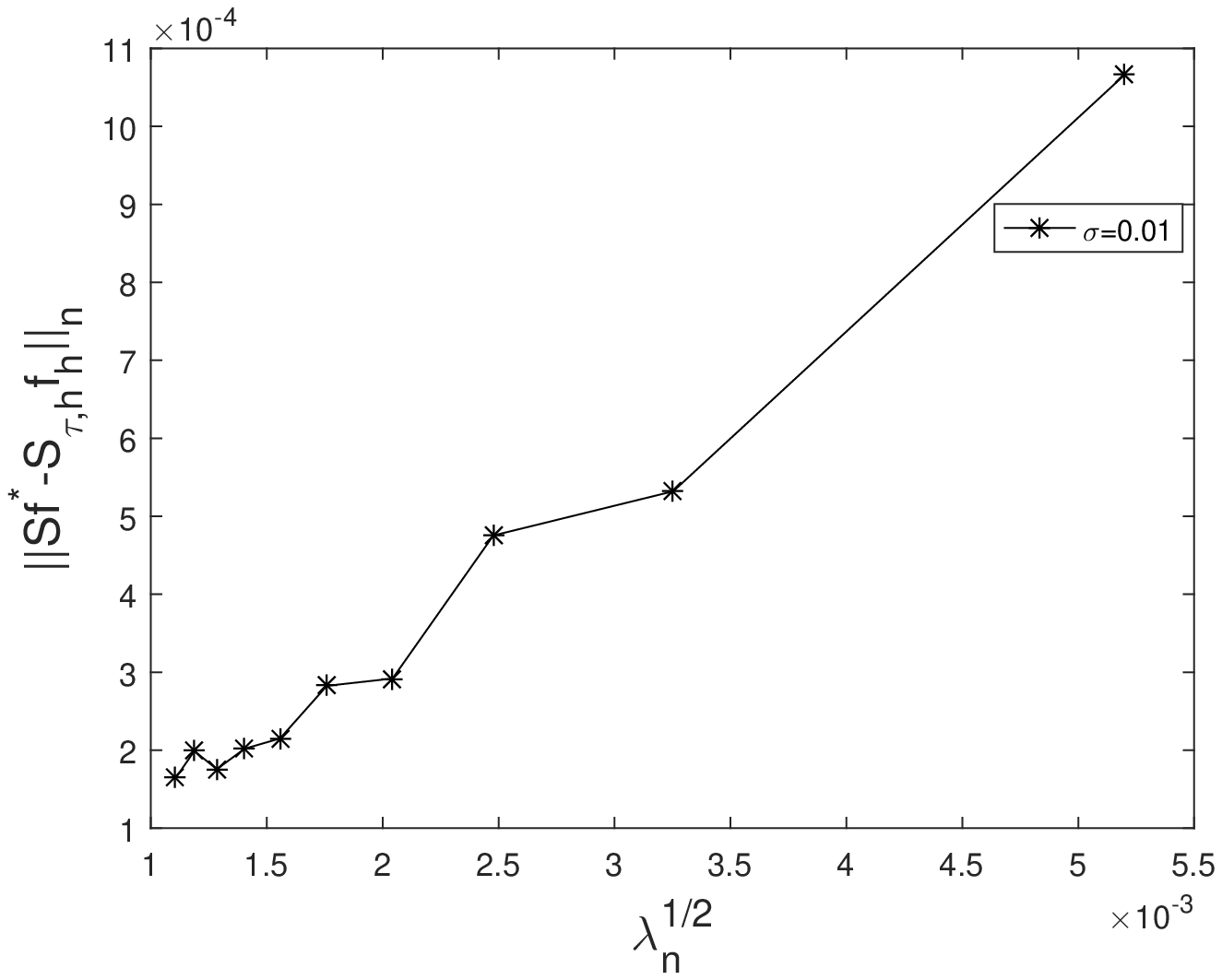} &
\includegraphics[width=6cm]{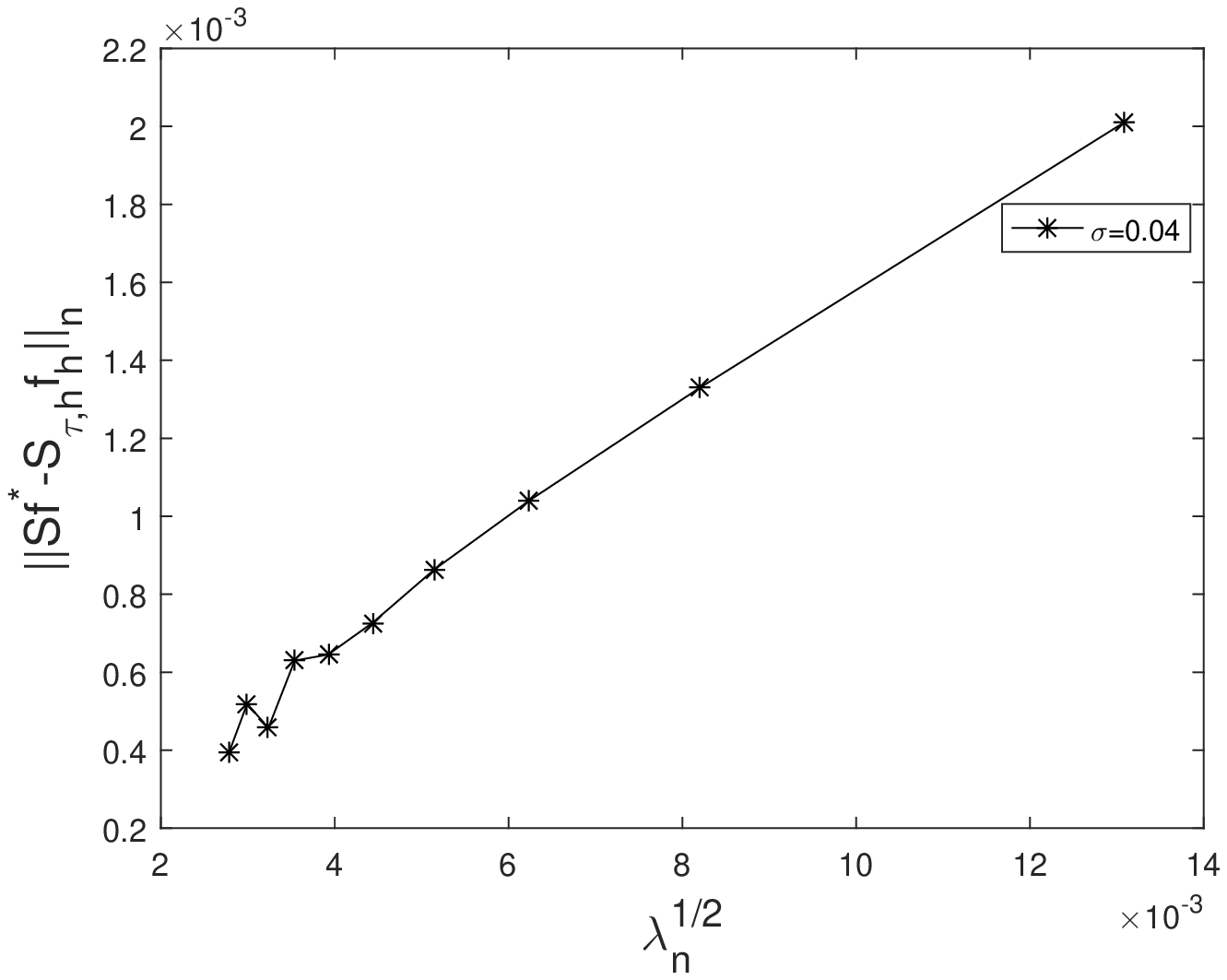} 
\end{tabular}
\end{center}
\caption{\em Linear dependence of the empirical error $\|Sf_*-S_{\tau,h} f_h\|_n$ 
on $\lambda_n^{1/2}$ with $\sigma=0.01$ (left) and $\sigma=0.04$ (right).}
\label{para-examexample.5}
\end{figure}

We can see from Figure \ref{para-examexample.5} clearly the linear dependence of the empirical error $\|Sf^*-S_{\tau,h} f_h\|_n$ on $\lam_n^{1/2}$ for $\sigma=0.01$ and $0.04$. We can compute that 
$\|Sf^*\|_{L^{\infty}(\Omega)}\approx 0.04$, so the relative noise levels $\sigma/\|Sf^*\|_{L^\infty(\Om)}$ are 
about $25\%$ and $100\%$ for $\sigma=0.01$ and $0.04$, respectively.

Through the previous 3 examples, we have verified the optimality of the choice rule \eqref{para-examo1} 
for $\lambda_n$, the stochastic convergence (Theorem \ref{examthm:4.2}),
and the convergence order of the finite element method. 
But we do not know the exact solution and the variance of the noise in most applications, 
so we use the next example to show the efficiency of Algorithm \ref{para-examj1} to determine an optimal 
regularization parameter $\lambda_n$ iteratively, without the knowledge of $f^*$ and $\sigma$. 

\begin{example} \label{para-examnumerical.4}
We choose $n=25\times 10^4$ and set the noise $e_1$, $\cdots$, $e_n$ in the dataset \eqref{eq:data} 
to be independent normal random variables with variance $\sigma=0.001$. 
Algorithm \ref{para-examj1} is terminated when the absolute difference between 
two consecutive iterates $\lambda_{n,k}$ and $\lambda_{n,k+1}$ is less than $10^{-10}$. 
\end{example}

We can compute that $\|Sf^*\|_{L^{\infty}(\Omega)}\approx 0.04$, so the relative noise level 
{\small $\sigma/\|Sf^*\|_{L^\infty(\Om)}$}
is about $2.5\%$ in this example. Figure \ref{para-examexample.4} shows clearly the convergence of 
the sequence $\{\lambda_{n,k}\}$ generated by Algorithm \ref{para-examj1}.
The numerical computation gives $\lambda_{n, 4}=5.53\times 10^{-8}$ that agrees very well 
with the optimal choice $5.33\times 10^{-8}$ given by \eqref{para-examo1}. 
Furthermore, $\|m-S_{\tau,h} f_h\|_n=9.99\times 10^{-4}$ provides also a good estimate of the variance $\sigma$. 

\begin{figure}[t]
\begin{center}
\begin{tabular}{cc}
\includegraphics[width=5cm]{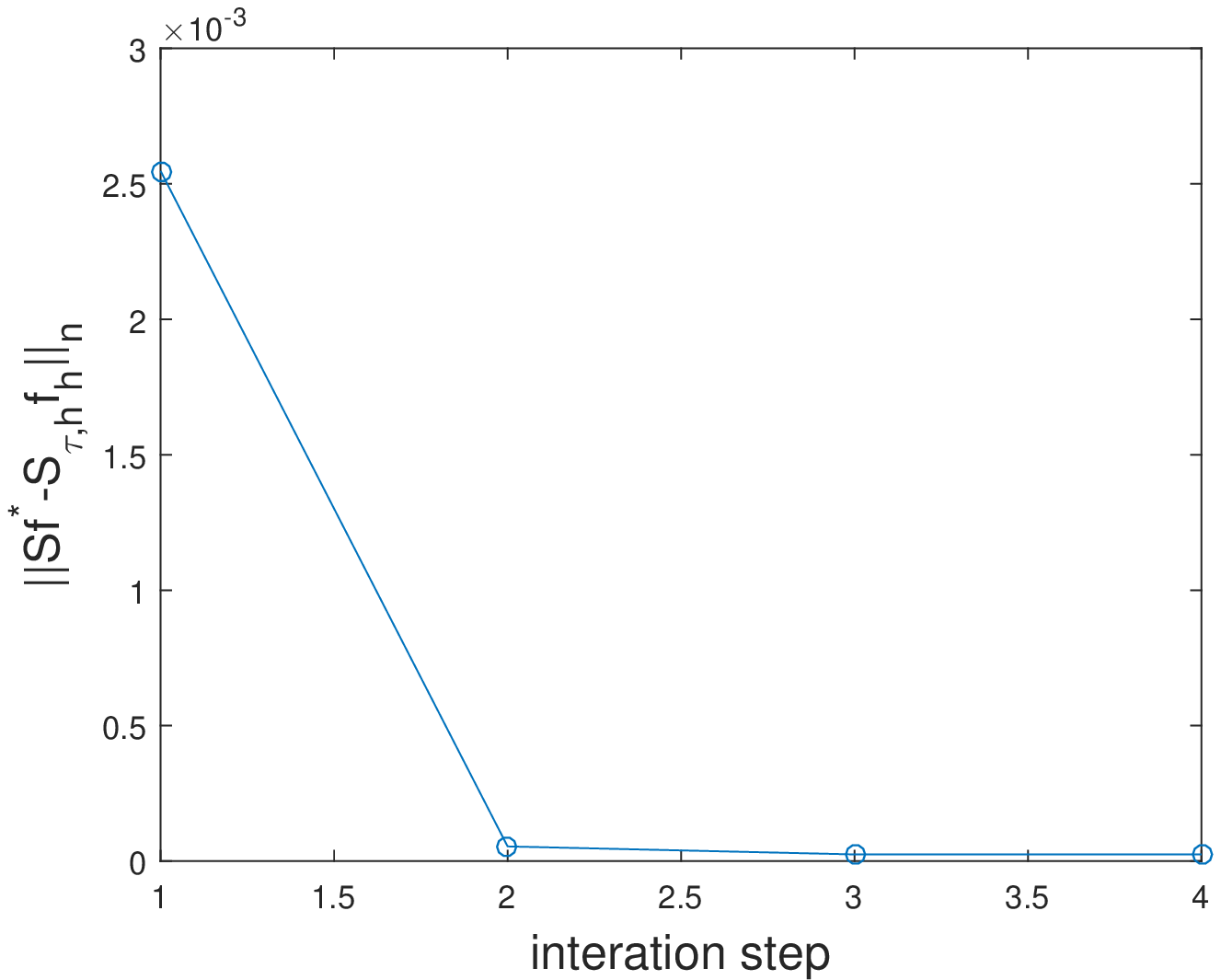} &
\includegraphics[width=5cm]{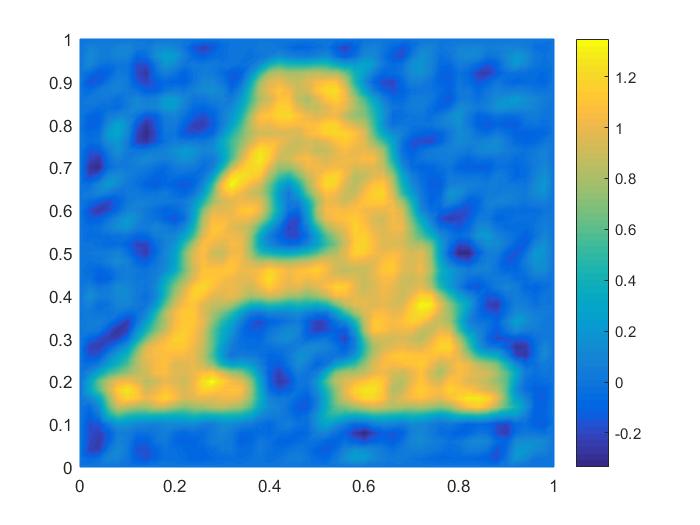}
\end{tabular}
\end{center}
\caption{\em The relative empirical error $\|Sf^*-S_{\tau,h} f_h\|_n$ at each iteration (left);
The computed solution $f_h$ at the end of iterations (right).}
\label{para-examexample.4}
\end{figure}


\end{document}